\documentclass[]{amsart} \usepackage{amsfonts}
\usepackage{bbding} \usepackage[dvips]{graphicx} \usepackage{amssymb,amsmath}
\usepackage{amsbsy,amscd} \usepackage{color}
\usepackage{enumerate}
 \usepackage{url}
\usepackage[colorlinks=true,
pdfstartview=FitV, linkcolor=blue, citecolor=red,
urlcolor=blue]{hyperref}

\newtheorem{theorem}{Theorem}[section] \newtheorem{proposition}[theorem]{Proposition}
\newtheorem{lemma}[theorem]{Lemma} \newtheorem{corollary}[theorem]{Corollary}

\theoremstyle{definition} \newtheorem{definition}[theorem]{Definition}
\newtheorem{example}[theorem]{Example}

\theoremstyle{remark} \newtheorem{remark}[theorem]{Remark} \numberwithin{equation}{section}
\numberwithin{figure}{section}

\newcommand{\Func}[1]{\mathcal{O}_{#1}}
\newcommand{\GL}{\mathrm{GL}}
\newcommand{\g}{\mathfrak{g}}

\newcommand{\GLR}{\mathrm{GL}_n\mathbb{R}}
\newcommand{\Hom}{\mathrm{Hom}}  
  
\newcommand{\Ad}{\mathrm{Ad}} 
\newcommand{\AD}{\mathit{Conj}}
\newcommand{\T}{\mathit{T}}
\newcommand{\Op}{\Theta}
\newcommand{\Hol}{\mathrm{Hol}}

\newcommand{\pr}{\mathrm{Pr}}
\newcommand{\e}{E}

\newcommand{\so}[1]{{#1}^\wedge}
\newcommand{\ta}[1]{{#1}^\vee}
\newcommand{\brac}[2]{\{{#1},{#2}\}}
\newcommand{\bi}[2]{({#1}\mid{#2})}
\DeclareMathOperator{\Tr}{Tr}

\begin{document}

\title[The quasi-Poisson Goldman formula]{The quasi-Poisson Goldman formula} \author[X. Nie]{Xin Nie} \address{Section de Math\'ematiques, 2-4 rue du Li\`evre, C.P. 64, 1211 Gen\`eve 4, Switzerland} \email{xin.nie@unige.ch}


\maketitle

\begin{abstract} We prove a quasi-Poisson bracket formula for the space of representations of the fundamental groupoid of a surface with boundary, which generalizes Goldman's Poisson bracket formula. We also deduce a similar formula for quasi-Poisson cross-sections. 
\end{abstract}

\tableofcontents

\section*{Introduction}
Let $\Sigma$ be a closed oriented surface and $G$ a Lie group with a fixed invariant scalar product on its Lie algebra. Ignoring a singular part, the moduli space $X_G(\Sigma)$ of flat connections on principal $G$-bundles over $\Sigma$, well known to be identified with $\Hom(\pi_1(\Sigma), G)/G$, is a symplectic manifold \cite{atiyah-bott, goldman_nature}. If $\Sigma$ is a bordered surface, i.e., has non-empty boundary, then the symplectic structure generalizes to a Poisson structure. 

For $\Sigma$ closed, Goldman \cite{goldman_invariant} provided a nice Poisson bracket formula for certain functions on $X_G(\Sigma)$.  When $G=\GL_n\mathbb{R}$, the formula leads to the so-called Goldman algebra of loops on a surface. The same result for $\partial\Sigma\neq\emptyset$, although widely believed to be true, was given a proof first in \cite{lawton}.

On the other hand, when $\partial\Sigma\neq\emptyset$, the quasi-Poisson theory of moduli spaces \cite{AMM, AKM} provides the following finite-dimensional construction of the Poisson structure on $X_G(\Sigma)$. We assume that $\partial\Sigma$ has $b$ components and choose a marked point $p_i$ on each of them. Consider the space of fundamental groupoid representations $$M_G(\Sigma):=\Hom(\pi_1(\Sigma; p_1,\cdots, p_b),G),$$ 
which has a natural $G^b$-action and the quotient is $M_G(\Sigma)/G^b=X_G(\Sigma)$. Then the space of functions on $M_G(\Sigma)$, denoted by $\Func{M_G(\Sigma)}$, is equipped with a canonical quasi-Poisson $G^b$-bracket $\{\cdot,\cdot\}_{M_G(\Sigma)}$ whose restriction to invariant functions $\Func{M_G(\Sigma)}^{G^b}=\Func{X_G(\Sigma)}$ gives the Poisson structure of $X_G(\Sigma)$. 

Our main result, Theorem \ref{theorem_quasibracket}, is a quasi-Poisson bracket formula for functions on $M_G(\Sigma)$, which generalizes Goldman's formula. When $G=\GL_n\mathbb{R}$, the formula leads to a quasi-Poisson version of the Goldman algebra, which was first discovered  by Massuyeau and Turaev \cite{massuyeau-turaev} using a different approach.

As a corollary, when $G$ is compact, we deduce a quasi-Poisson bracket formula (Theorem \ref{theorem_crossbracket}) for functions on the cross-section
$$
L=\bigcap_{i=1}^b\mu_i^{-1}(U)\subset M_G(\Sigma),
$$
where $\mu_i: M_G(\Sigma)\rightarrow G$ is the holonomy of the $i$-th component of $\partial\Sigma$ (with reversed orientation), and $U\subset G$ is a certain cross-section of the conjugation $G$-action on itself, so that $L$ is a Hamiltonian quasi-Poisson $H^b$-manifold for a subgroup $H\subset G$, see \S \ref{subsection_cross}. 

The present work was motived by a relationship between $L$ and some particular rational functions on $X_G(\Sigma)$ in the case $G=\mathrm{SL}_n(\mathbb{R})$ \cite{fock-goncharov, labourie_swapping}. This aspect will appear in the author's thesis \cite{nie_these} and a forthcoming paper.

As this paper was being prepared for release, Li-Bland and \v Severa independently released the paper \cite{libland-severa}, whose results overlap with ours.

\subsection*{Acknowledgements} 
The author is grateful to Anton Alekseev, David Li-Bland and Pavol \v{S}evera for their interest in this work and useful discussions, and to his thesis advisor Gilles Courtois for carefully reading the first draft of this paper and valuable comments.

\section{Quasi-Poisson theory}\label{section_quasipoissontheory}
In this preliminary section we recall a version of the quasi-Poisson theory \cite{AMM, AKM}.  The presentation here can be viewed as a simplified version of a more general framework \cite{severa, libland-severa}.

\subsection{Fundamental groupoids of surfaces}\label{subsection_surface}  
Let $\Sigma$ be a compact oriented surface such that $\partial \Sigma$ has $b\geq 1$ boundary components and a marked point $p_i$ is chosen on each of them. But a \emph{path} on $\Sigma$ we always mean an oriented smooth curve whose starting and ending points belong to $\{p_1,\cdots, p_b\}$. Let 
$$\pi_1(\Sigma):=\pi_1(\Sigma; p_1,\cdots, p_b).$$ 
denote the fundamental groupoid of $\Sigma$, i.e., the set of (end-points-fixing) homotopy classes of paths, equipped with the obvious partial multiplication. 

Fix a Lie group $G$. A \emph{representation} of $\pi_1(\Sigma)$ in $G$ is by definition a map $\pi_1(\Sigma)\rightarrow G$ preserving multiplications. Let the space of all representations be denoted by $$M_G(\Sigma):=\Hom(\pi_1(\Sigma),G).$$ 

If $\alpha$ is a path, the \emph{holonomy} along $\alpha$ is the map
$$\Hol_\alpha: M_G(\Sigma)\rightarrow G,\quad m\mapsto m(\alpha).$$

Let $\beta_i$ denote the $i$-th boundary loop (with induced orientation). The $i$-th \emph{reversed boundary holonomy} $$\mu_i:=\Hol_{\beta_i^{-1}}$$
will play a special role later on.


There is a natural $G^b$-action on $M_G(\Sigma)$ given by
$$
((g_1,\cdots, g_b).m)(\alpha)=g_im(\alpha)g_j^{-1}$$
for any $m\in M_G(\Sigma)$ and path $\alpha$ going from $p_i$ to $p_j$. 
In other words, if $\alpha$  starts and ends both at $p_i$ (resp. only starts/ends at $p_i$), then $\Hol_\alpha$ is equivariant with respect to the $i$-th $G$-action on $M_G(\Sigma)$ and the $G$-action on itself by conjugation (resp. left/right multiplication).

It is easy to see that the map
$$
M_G(\Sigma)\rightarrow \Hom(\pi_1(\Sigma; p_1),G)
$$
induced by the injection $\pi_1(\Sigma; p_1)\rightarrow \pi_1(\Sigma)$ is a principle $G^{b-1}$-bundle, hence the quotient $M_G(\Sigma)/G^b$ is identified with $X_G(\Sigma)=\Hom(\pi_1(\Sigma; p_1),G)/G$.

If $\Sigma=\Sigma_{g,b}$ is the connected surface with genus $g$
and $b$ boundary components, a set of $2(b-1)+2g$ generators of $\pi_1(\Sigma)$ without relations can be constructed
as follows. Take a path $\alpha_i$ from $p_1$ to $p_i$ for each $2\leq i\leq b$, and two loops
$\gamma_j, \delta_j$ based at $p_1$ for each $1\leq j\leq g$, such that by cutting the surface along
these paths we obtain a polygon whose edges are successively
$$ 
\beta_1, \alpha_2, \beta_2, \alpha_2^{-1},\cdots,
\alpha_b, \beta_b, \alpha_b^{-1},\gamma_1, \delta_1, \gamma_1^{-1}, \delta_1^{-1},\cdots,\gamma_g,
\delta_g,\gamma_g^{-1}, \delta_g^{-1},
$$
see the second picture of Figure \ref{figure_splitting}. Then $\{\alpha_i, \beta_i, \gamma_j, \delta_j\}$ form the required generators.


 As a consequence, put $$u_i=\Hol_{\alpha_i}, \ v_i=\Hol_{\beta_i},\  a_j=\Hol_{\gamma_j},\  b_j=\Hol_{\delta_j},$$ then
 $(u_i, v_i, a_j, b_j)$ form a $G$-valued coordinates system of $M_G(\Sigma)$, which identifies $M_G(\Sigma)$ with $G^{2(b-1)+2g}$. Clearly, $\mu_i=v_i^{-1}$ ($2\leq i\leq b$) and $$\mu_1=u_2v_2u_2^{-1}\cdots u_bv_bu_b^{-1}[a_1,b_1]\cdots[a_g,b_g],$$
where $[a,b]=aba^{-1}b^{-1}$ is the commutator.
 
Here are two simplest examples
\begin{example}[$\Sigma_{0,2}$, the cylinder]\label{example_annulus}
The coordinates system $(u,v)=(u_2,v_2)$ identifies $M_G(\Sigma_{0,2})$ with $G^2$, on which the $G^2$ action reads 
$$(u,v)\overset{(g_1,g_2)}\longmapsto (g_1ug_2^{-1},g_2vg_2^{-1}),$$
and we have
$$
\mu_1(u,v)=uvu^{-1}, \mu_2(u,v)=v^{-1}.
$$
\end{example}

\begin{example}[$\Sigma_{1,1}$, the one-holed torus]\label{example_oneholedtorus}
The coordinates system $(a,b)=(a_1,b_1)$ identifies $M_G(\Sigma_{1,1})$ with $G^2$.  
The canonical $G$-action is the conjugation action on both factors. Put $\mu=\Hol_{\beta^{-1}}$ where $\beta$ is the boundary loop,  then $\mu(a,b)=[a,b]$.
\end{example}

%
%

\subsection{Splitting}\label{subsection_splitting} By a \emph{splitting arc} on
$\Sigma$ we mean an embedded segment $\Delta$ which joins some marked point, say, $p_1$, with another point on the same boundary component $\beta_1$.

Splitting $\Sigma$ along $\Delta$ and splitting $p_1$ into two marked points $p_{1}^L$ and $p_{1}^R$, we get a surface $\Sigma_\Delta$ with $b+1$ boundary components, each of them still having a marked point. See the following picture.
\begin{figure}[h] 
\centering \includegraphics[width=3in]{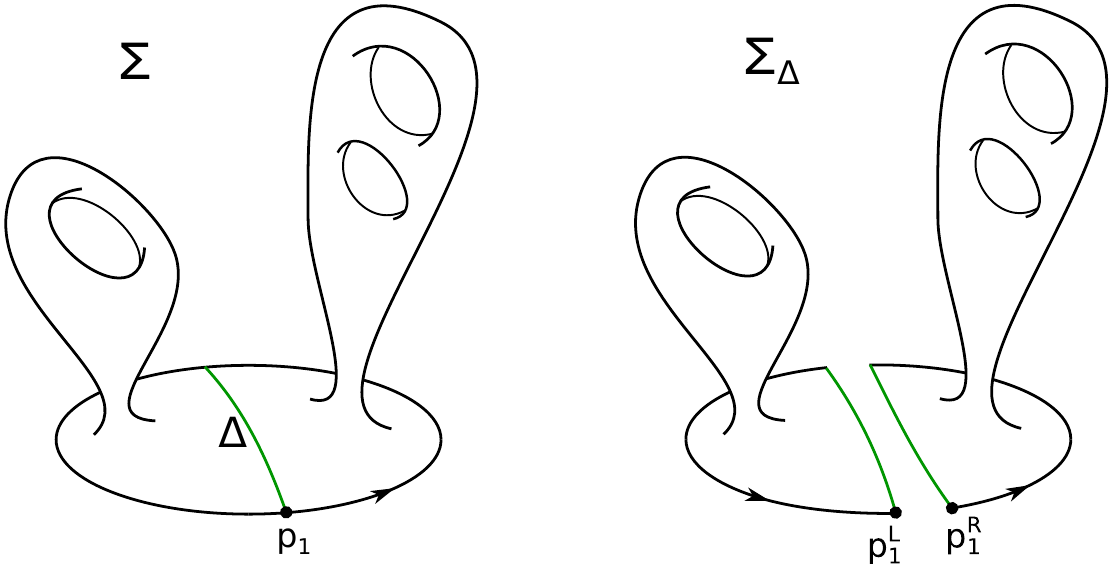} 
\end{figure}

\begin{remark}
The two particular marked points $p_1^L, p_1^R\in\partial\Sigma_\Delta$ have a left-right distinction. Indeed, identifying a neighborhood $U$ of $p_1$ in $\Sigma$ with the upper half-plan in an orientation-preserving manner, we assume that after splitting $U$ into two,  $p_1^L$ is on the left half and $p_1^R$ on the right.
\end{remark}

The gluing map $\Sigma_\Delta\rightarrow\Sigma$ induces a map between fundamental groupoids, and hence a map between representation spaces
$$R_\Delta:M_G(\Sigma)\longrightarrow M_G(\Sigma_\Delta). $$

By a suitable adaptation of the Van Kampen theorem, one can prove that $R_\Delta$ is bijective. Identifying $M_G(\Sigma)$ and $M_G(\Sigma_\Delta)$ via $R_\Delta$, it is easy to see that 
\begin{itemize}
\item The $G^b$-action on $M_G(\Sigma)$ is induced by the $G^{b+1}$-action on $M_G(\Sigma_\Delta)$ and the diagonal embedding of the first factor
$$
G^b\hookrightarrow G^{b+1},\quad (g_1,g_2, \cdots, g_b)\mapsto (g_1,g_1, g_2,\cdots, g_b).
$$
\item Consider the maps $\mu_1:M_G(\Sigma)\rightarrow G$ and 
$$
\mu_1^L:=\Hol_{(\beta_1^L)^{-1}}, \  \mu_1^R:=\Hol_{(\beta_1^R)^{-1}}: M_G(\Sigma_\Delta)\longrightarrow G,
$$ 
where $\beta_1^L$ (resp. $\beta_1^R$) is the boundary loop of $\Sigma_\Delta$ at $p_1^L$ (resp. $p_1^R$), then $$\mu_1=\mu_1^L\cdot\mu_1^R.$$
Here and below, the product of two maps $M\rightarrow G$ is pointwise defined using the product in $G$.
\end{itemize}

\begin{example}\label{example_splittingoneholedtorus}
There is a splitting arc $\Delta$ on $\Sigma=\Sigma_{1,1}$ such that the standard generators $\gamma, \delta$ of $\pi_1(\Sigma_{1,1})$ becomes standard generators of $\Sigma_\Delta=\Sigma_{0,2}$, see the first picture of Figure \ref{figure_splitting}. With the coordinates of Example \ref{example_annulus} and \ref{example_oneholedtorus}, we have
$$
\begin{array}{rcl}
R_\Delta:M_G(\Sigma_{1,1})&\overset\sim\rightarrow& M_G(\Sigma_{0,2})\\
(a,b)&\mapsto &(u,v)=(a,b).
\end{array}
$$
Moreover, we have $\mu=\mu_1\cdot\mu_2$.
\begin{figure}[h] 
\centering 
\includegraphics[width=1.8in]{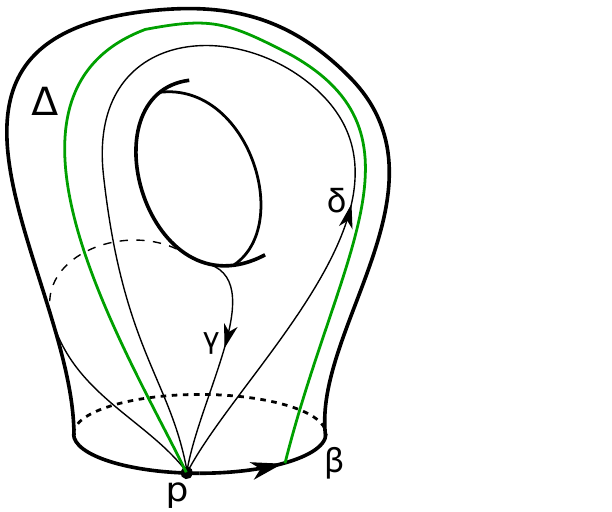} 
\includegraphics[width=2.1in]{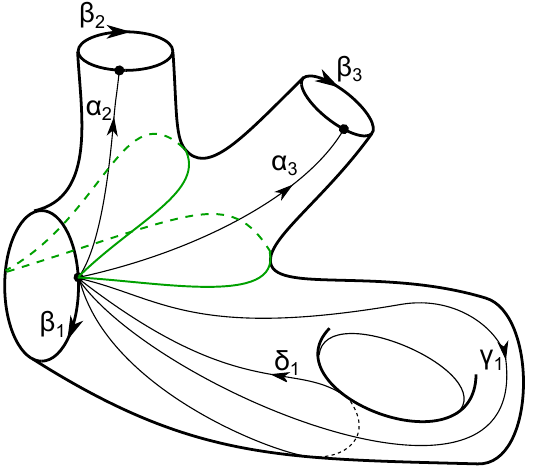} 
\caption{Splitting $\Sigma_{1,1}$ and $\Sigma_{g,b}$}\label{figure_splitting} 
\end{figure}
\end{example}
\begin{example}\label{example_splitting}
Let $\Sigma=\Sigma_{g,b}$ and take paths $\alpha_i$ and $\delta_j,\gamma_j$ as in \S \ref{subsection_surface}. Splitting $\Sigma$ successively along the $b-2+g$ splitting arcs shown in the second picture of Figure \ref{figure_splitting}, we get a surface $\hat{\Sigma}$ which is the disjoint union of $b-1$ copies of $\Sigma_{0,2}$ and $g$ copies of $\Sigma_{1,1}$, and a bijection between the representation spaces
$$
M_G(\Sigma_{g,b})\overset\sim\longrightarrow M_G(\hat{\Sigma})=M_G(\Sigma_{0,2})^{b-1}\times M_G(\Sigma_{1,1})^g.
$$
\end{example}

\subsection{Quasi-Poisson manifolds}
\begin{definition}\label{definition_quasipoisson}
Let $G$ be a Lie group with an invariant scalar product $\bi{\cdot}{\cdot}$ on the Lie algebra $\g$. Let $M$ be a manifold with a $G$-action denoted by 
$\rho$. A $G$-invariant bivector field $P\in\Gamma(\bigwedge^2\T M)$ is called a \emph{quasi-Poisson $G$-tensor} and $(M, P)$ called a \emph{quasi-Poisson $G$-manifold} if $P$ satisfies
\begin{equation}\label{equation_defquasipoisson}
[P,P]=\rho_\phi.
\end{equation}
Here $[P,P]\in \Gamma(\bigwedge^3\T M)$ is the Schouten bracket, $\phi\in(\bigwedge^3\mathfrak{g})^\g$ is a canonical invariant trivector associated with $\bi{\cdot}{\cdot}$, which has the following expression in terms of an orthonormal basis $(e_i)$ of $\g$:
$$
\phi=\frac{1}{12}\sum_{i,j,k}\bi{e_i}{[e_j, e_k]}e_i\wedge e_j\wedge e_k,
$$
and we have extended the Lie algebra homomorphism 
$$\g\rightarrow\Gamma(\T M),\quad x\mapsto \rho_x,\mbox{ where } \rho_x(m):=\frac{d}{dt}\big|_{t=0}\rho_{\exp(-tx)}(m)$$
to a map $\bigwedge^\bullet\g\rightarrow\Gamma(\bigwedge^\bullet\T M)$ preserving Schouten brackets.

The \emph{quasi-Poisson bracket} $\brac{\cdot}{\cdot}$ associated to $P$ is defined as the anti-symmetric bilinear form on $\Func{M}$ given by 
$$
\brac{f}{g}=P(df,dg).
$$


\end{definition}

\begin{definition}\label{definition_moment}
Let $(M,P)$ be a quasi-Poisson $G$-manifold, where the $G$-action is denoted by $\rho$. An equivariant map $\mu:M\rightarrow G$ (where $G$ acts on itself by conjugation) is called a \emph{(group-valued) moment map}, and $(M, P, \mu)$ called a \emph{Hamiltonian quasi-Poisson manifold}, if $\mu$ satisfies
\begin{equation}\label{equation_diffofhol}
\mu^*\theta(P^\sharp(df))=-\frac{1}{2}(1+\Ad_\mu^{-1})\chi_f
\end{equation}
for any function $f\in\Func{M}$. Here $\theta\in\Omega^1(G, \g)$ is the left invariant Maurer-Cartan $1$-form, $P^\sharp: \T^*M\rightarrow \T M$ is defined by $P^\sharp(df)=P(df,\cdot)$ and the \emph{variation map} $\chi_{f}: M\rightarrow \g$ of $f$ is defined by 
\begin{equation}\label{equation_variation}
\bi{\chi_f(m)}{x}=\rho_x(f)(m),\quad\mbox{for any }m\in M, x\in\g.
\end{equation}
\end{definition}

\begin{remark}
A quasi-Poisson $G$-manifold $(M, P)$ is said to be \emph{non-degenerate} if at any point $m\in M$ we have $\T_mM=P_m^\sharp(\T^*M)+\rho_\g(m)$. Any Hamiltonian quasi-Poisson manifold is foliated by non-degenerate ones. 

In the present paper we mainly work with the canonical quasi-Poisson structure on $M_G(\Sigma)$, which is known to be non-degenerate. 
\end{remark}

The following fusion construction provides a way to get new quasi-Poisson manifolds from old ones:
\begin{definition}\label{definition_fusion}
Let $(M,P)$ be a quasi-Poisson $G\times G\times H$-manifold, where the $G\times G$-action is denoted by $\rho$. Let $(e_i)$ be an orthonormal basis of $\g$ and put
$$
\psi=\frac{1}{2}\sum_i (e_i,0)\wedge(0,e_i)\in\wedge^2(\g\oplus\g).
$$
Then the bivector field
\begin{equation}\label{equation_fusion}
P'=P-\rho_\psi
\end{equation}
is a quasi-Poisson $G\times H$-tensor, where $G\times H$ acts on $M$ via the diagonal embedding
$$
G\times H\hookrightarrow G\times G\times H,\quad (g,h)\mapsto (g,g,h).
$$
The quasi-Poisson $G\times H$-manifold $(M, P')$ is called the \emph{fusion} of $(M, P)$.

Furthermore, if $\mu=(\mu_1,\mu_2,\nu): M\rightarrow G\times G\times H$ is a moment map for $P$, then $\mu'=(\mu_1\cdot\mu_2,\nu): M\rightarrow G\times H$ is a moment map for $P'$. The Hamiltonian quasi-Poisson $G\times H$-manifold $(M, P', \mu')$ is also called the fusion of $(M, P, \mu)$.

In particular, let $(M_i, P_i, (\mu_i,\nu_i))$ be a Hamiltonian quasi-Poisson $G\times H_i$-manifold ($i=1,2$). Clearly, $(M_1\times M_2, P_1+P_2, (\mu_1,\mu_2, \nu_1, \nu_2))$ is a Hamiltonian quasi-Poisson $G\times G\times H_1\times H_2$-manifold. The fusion $(M_1\times M_2, P', (\mu_1\cdot\mu_2, \nu_1, \nu_2))$ is called the \emph{fusion product} of $M_1$ and $M_2$, denoted by $M_1\circledast M_2$.
\end{definition}


\begin{example}\label{example_G}
Let $G\times G$ acts on $G$ by $\rho_{(g, h)}(a)=gah^{-1}$. Then the trivial bivector field $P=0$ on $G$ is a quasi-Poisson $G\times G$-tensor. It does not admit moment maps.
\end{example}

\begin{example}
Applying fusion to Example \ref{example_G}, we get a quasi-Poisson $G$-tensor on $G$ with respect to the conjugation action:
$$
P_G=\frac{1}{2}\sum_ie_i^R\wedge e_i^L.
$$
Here $e_i^L$ (resp. $e_i^R$) is the left (resp. right) invariant vector field on $G$ generated by $e_i\in\g=\T_eG$.
It can be shown that the identity $G\rightarrow G$ is a moment map for $P_G$.
\end{example}

\begin{example}\label{example_double}
The product of two copies of Example \ref{example_G} is $M=G\times G$ with the $G^4$-action $\rho_{(g_1,g_2,g_3,g_4)}(a,b)=(g_1ag_2^{-1},g_3bg_4^{-1})$. Applying fusions with respect to the first and last factor of $G^4$, then with respect to the second and third factor, we get the following quasi-Poisson $G\times G$-tensor on $G\times G$ (where the action is $\rho_{(g,h)}(a,b)=(gah^{-1}, hbg^{-1})$):
$$
P=\frac{1}{2}\sum_i(e_i^{1,L}\wedge e_i^{2,R}+e_i^{1,R}\wedge e_i^{2,L}).
$$
Here $e_i^{1,L}$ (resp. $e_i^{2,L}$) is the left invariant vector field on the first (resp. second) factor generated by $e_i$.

It can be shown that $$\mu: G\times G\rightarrow G\times G, (a,b)\mapsto (ab, a^{-1}b^{-1})$$ is a moment map in this case. This Hamiltonian quasi-Poisson is called the \emph{double} of $G$, denoted by $D(G)$.

Applying fusion again to $D(G)$, we get a quasi-Poisson $G$-tensor on $G\times G$ (where $G$-acts by conjugation on both factors), which has moment map $(a,b)\mapsto aba^{-1}b^{-1}$. This Hamiltonian quasi-Poisson manifold is called the \emph{fused double}.
\end{example}


\subsection{The Quasi-Poisson structure on $M_G(\Sigma)$}\label{subsection_qh}
A main result in quasi-Poisson theory is that there is a canonical quasi-Poisson $G^b$-tensor on $M_G(\Sigma)$, whose reduction gives the standard Poisson structure on $X_G(\Sigma)$. 

First let us consider the simplest case $\Sigma=\Sigma_{0,2}$. As a $G^2$-manifold, $M_G(\Sigma_{0,2})$ is identified with the double $D(G)$ through
\begin{align*}
M_G(\Sigma_{0,2})&\overset\sim\longrightarrow D(G)\\
(u,v)&\longmapsto(a, b)=(u, vu^{-1}).
\end{align*}
Thus $M_G(\Sigma_{0,2})$ is a Hamiltonian quasi-Poisson $G^2$-manifold, with moment map
$$(ab, a^{-1}b^{-1})=(uvu^{-1}, v^{-1})=(\mu_1,\mu_2).$$

Next, under the splitting map $R_\Delta: M_G(\Sigma_{1,1})\overset\sim\rightarrow M_G(\Sigma_{0,2})\cong D(G)$ of Example \ref{example_splittingoneholedtorus}, 
the $G$-action on $M_G(\Sigma_{1,1})$ coincides with the action of diagonal subgroup of $G^2$ on $D(G)$, hence we can also endow $M_G(\Sigma_{1,1})$ with the Hamiltonian quasi-Poisson structure of the fused double. The moment map $\mu_1\cdot\mu_2$ coincides with $\mu: M_G(\Sigma_{1,1})\rightarrow G$.

Similarly, in view of Example \ref{example_splitting}, for any $\Sigma=\Sigma_{g,b}$ we can endow $M_G(\Sigma)$ with the Hamiltonian quasi-Poisson $G^b$-manifold structure of the fusion product
$$
M_G(\Sigma)=\underbrace{M_G(\Sigma_{0,2})\circledast \cdots\circledast M_G(\Sigma_{0,2})}_{b-1}\circledast \underbrace{M_G(\Sigma_{1,1})\circledast \cdots\circledast M_G(\Sigma_{1,1})}_{g},
$$
and the moment map is 
$$
(\mu_1,\cdots,\mu_b): M_G(\Sigma)\longrightarrow G^b.
$$

\begin{theorem}\label{theorem_AMM} 
The above Hamiltonian quasi-Poisson structure on $M_G(\Sigma)$ does not depend on the way we split $\Sigma$ into pieces, and the restriction of the quasi-Poisson bracket to $\Func{X_G(\Sigma)}=\Func{M_G(\Sigma)}^{G^b}\subset\mathcal{O}_{M_G(\Sigma)}$ is the standard Poisson structure on $X_G(\Sigma)$. 

Moreover, this quasi-Poisson structure has the following property: let $\Delta\subset\Sigma$ be a splitting arc issuing from the marked point $p_1\in\partial\Sigma$, then via the identification $R_\Delta: M_G(\Sigma)\overset\sim\rightarrow M_G(\Sigma_\Delta)$ (see \S \ref{subsection_splitting}), the Hamiltonian quasi-Poisson manifold $M_G(\Sigma)$ is the same as the fusion of $M_G(\Sigma_\Delta)$ with respect to the first two factors of $G^{b+1}$.
\end{theorem}

Briefly speaking, the theorem comes from the fact that the standard Poisson structure on $X_G(\Sigma)$ arises as a reduction of the infinite-dimensional symplectic manifold $\mathcal{N}_G(\Sigma)$ of flat connections, while the quasi-Poisson structure on $M_G(\Sigma)$ is a partial reduction. However, the first statement of the theorem can be shown in a straightforward way, see \cite{libland-severa}.

\section{The quasi-Poisson bracket on $M_G(\Sigma)$}\label{section_quasi}
The aim of this section is to prove the main result of the paper, Theorem \ref{theorem_quasibracket}, and deduce some simple corollaries. First we need some more notations.

%
%

\subsection{Notations}\label{subsection_notations}
\begin{definition}\label{definition_variation}
Let $G$ be a Lie group with an invariant scalar product $\bi{\cdot}{\cdot}$ on its Lie algebra $\g$. For any $\Phi\in\Func{G}$, we define maps $\so{\Phi}, \ta{\Phi}:G\rightarrow\mathfrak{g}$ by
$$
x^R(\Phi)(g)=\bi{\so{\Phi}(g)}{x},\quad x^L(\Phi)(g)=\bi{\ta{\Phi}(g)}{x},
$$
for any $x\in\g$ and $g\in G$.
\end{definition}

We will often make use of the following characterizing property of $\ta{\Phi}$ and $\so{\Phi}$: let $\theta\in\Omega^1(G, \g)$ (resp. $\bar\theta\in\Omega^1(G, \g)$) denote the left (resp. right) invariant Maurer-Cartan $1$-form, then for any manifold $M$ and map $u:M\rightarrow G$, we have 
$$
d(\Phi(u))=\bi{\so{\Phi}(u)}{u^*\bar\theta}=\bi{\ta{\Phi}(u)}{u^*\theta}.
$$
Here and below, $\Phi(u)$ and $\Phi^I(u)$ means the composition of $u$ with $\Phi$ and $\Phi^I$ ($I=\wedge, \vee$), respectively. 

\begin{example}[Matrix entry functions]\label{example_caseGL}
Put $G=\GLR$ and equip the Lie algebra $\mathfrak{g}=\mathfrak{gl}_n\mathbb{R}$ with the $\Ad$-invariant scalar product $\bi{x}{y}=\Tr(xy)$. Let $\Phi_{ij}\in\Func{G}$ $(1\leq i,j\leq n)$ be the \emph{matrix entry function} given by $\Phi_{ij}(g)=g_{ij}$ (the $(i,j)$-entry of $g$). 

Let $\e_{ij}$ be the $n\times n$ matrix whose $(i,j)$-entry is $1$ and all other entries vanish. The definition of  $\ta{\Phi}_{ij}$ yields
$$
\bi{\ta{\Phi}_{ij}(g)}{x}=d\Phi_{ij}(x^L(g))=d\Phi_{ij}(gx)=(gx)_{ij}=\Tr(\e_{ji}gx)
=\bi{\e_{ji}g}{x},
$$
from which we get
$$
\ta{\Phi}_{ij}(g)=\e_{ji}g,
$$
and similarly
$$
\so{\Phi}_{ij}(g)=g\e_{ji}.
$$
\end{example}

Here are some elementary properties of $\ta{\Phi}$ and $\so{\Phi}$.
\begin{lemma}\label{lemma_taso}
\begin{enumerate}[(i)]
\item\label{item_tasoi}
$\ta{\Phi}$ and $\so{\Phi}$ are related by $$\Ad_g\ta{\Phi}(g)=\so{\Phi}(g),\quad\forall g\in G.$$
\item\label{item_tasoii}
Let $\hat\Phi$ be the function on $G$ given by $\hat\Phi(g)=\Phi(g^{-1})$, then for any $g\in G$ we have
$$
\ta{\hat\Phi}(g)=-\so{\Phi}(g^{-1}), \quad \so{\hat\Phi}(g)=-\ta{\Phi}(g^{-1}).
$$
\item\label{item_tasoiii}
Fix $a,b\in G$ and let $\widetilde\Phi$ be the function on $G$ given by $\widetilde\Phi(g)=\Phi(agb)$, then for any $g\in G$ we have
$$
\ta{\widetilde\Phi}(g)=\Ad_b\ta{\Phi}(agb),\quad \so{\widetilde\Phi}(g)=\Ad_{a^{-1}}\so{\Phi}(agb).
$$
\item\label{item_tasoiv}
Let $M$ be a $G$-manifold and $u: M\rightarrow G$ be a map. We denote respectively by $L$, $R$ and $\AD$ the left, right and conjugation $G$-action on itself, then
$$
\chi_{\Phi(u)}=\left\{
\begin{array}{ll}
-\so{\Phi}(u)&\mbox{if $u$ is $L$-equivariant},\\
\ta{\Phi}(u)&\mbox{if $u$ is $R$-equivariant},\\
-\so{\Phi}(u)+\ta{\Phi}(u)&\mbox{if $u$ is $\AD$-equivariant}.
\end{array}
\right.
$$
(Recall that for $f\in\Func{M}$ the variation map $\chi_f: M\rightarrow \g$ is defined by (\ref{equation_variation}).)
\end{enumerate}
\end{lemma}

We also need some notations concerning paths on $\Sigma$ (see \S \ref{subsection_surface} for assumptions on $\Sigma$). 

Given $\Phi\in\Func{G}$ and a path $\alpha$,  we denote 
$$\Phi_\alpha:=\Phi(\Hol_\alpha)\in\Func{M_G(\Sigma)},\quad \Phi^I_\alpha:=\Phi^I(\Hol_\alpha): M_G(\Sigma)\rightarrow \mathfrak{g}\quad (I=\wedge, \vee).$$

The \emph{starting direction}  (resp. \emph{ending direction}) of $\alpha$, denoted by $\so{\alpha}$ (resp. $\ta{\alpha}$) is the tangent vector at the starting (resp. ending) point of $\alpha$ up to positive scaling. 

If $p\in\partial\Sigma$ is a marked point, we use the notation ``$\alpha^I\vdash p$" (where $I=\vee,  \wedge$) to indicated that ``$\alpha^I$ lies on $p$", i.e., $\alpha$ starts from $p$ if $I=\wedge$ and $\alpha$ ends at $p$ if $I=\vee$.
  
Two paths $\alpha$ and $\beta$ are said to be \emph{in general position} if their interior intersection points are transversal double points (called  \emph{crossings}), and $\so{\alpha}$, $\ta{\alpha}$, $\so{\beta}$, $\ta{\beta}$ do not coincide with each other.


Let $\alpha\#\beta:=(\alpha\cap\beta)\setminus\partial\Sigma$ denote the set of crossings. For any $q\in\alpha\#\beta$, we let
$\varepsilon_q(\alpha,\beta)=\pm 1$ be the oriented intersection number of $\alpha$ and $\beta$ at $q$, and let $\alpha*_q\beta$ denote the path which starts at $\so{\alpha}$, runs along $\alpha$ before $q$, then switches to $\beta$ at $q$, runs along $\beta$ and ends at $\ta{\beta}$.

For any two symbols $I,J=\wedge, \vee$, we also let $\varepsilon(\alpha^I,\beta^J)=0,\pm \frac{1}{2}$ denote the ``oriented intersection number" of $\alpha^I$ and $\beta^J$, namely, define $\varepsilon(\alpha^I,\beta^J)=0$ if $\alpha^I$ and $\beta^J$ do not lie on the same marked point; otherwise both $\alpha^I$ and $\beta^J$ are in $\T_p\Sigma$ for a marked point $p\in\partial\Sigma$, and we define $\varepsilon(\alpha^I,\beta^J)=\frac{1}{2}$ if the frame $(\alpha^I,\beta^J)$ is compatible with the orientation of $\Sigma$.

Finally, we define the algebraic intersection number of $\alpha$ and $\beta$ as 
$$
i(\alpha,\beta):=\sum_{I,J=\wedge, \vee}\varepsilon(\alpha^I,\beta^J)+\sum_{q\in\alpha\#\beta}\varepsilon_q(\alpha,\beta).
$$
This generalizes the usual notion of algebraic intersection number for closed curves. It can be shown that the new notion is invariant under (endpoints-fixing) homotopy.

Let us remark that Lemma \ref{lemma_taso} (\ref{item_tasoiv}) has the following formulation when $M=M_G(\Sigma)$ and $u=\Hol_\alpha$, which will be used a few times below. We denote the variation map of $f\in\Func{M_G(\Sigma)}$ with respect to the $G$-action associated to a marked point $p\in\{p_1,\cdots,p_b\}$ by $\chi^{p}_f:M_G(\Sigma)\rightarrow\mathfrak{g}$. 
\begin{lemma}\label{lemma_formulation}
Put $\varepsilon_{IJ}=1$ if $I=J$ and $\varepsilon_{IJ}=-1$ if $I\neq J$. Let $\Phi\in\Func{G}$, and $\alpha$ be an   path on $\Sigma$. Then for any marked point $p\in\partial \Sigma$ we have
$$
\chi_{\Phi_\alpha}^{p}=\sum_{I:\alpha^I\vdash p}\varepsilon_{I\vee}\Phi^I_\alpha.
$$
\end{lemma}

\subsection{The quasi-Poisson bracket formula}\label{subsection_quasibracket}
\begin{theorem}\label{theorem_quasibracket}
For any $\Phi,\Psi\in\Func{G}$ and paths $\alpha$, $\beta$ in general position. The quasi-Poisson bracket of the functions $\Phi_\alpha$ and $\Psi_\beta$ on $M_G(\Sigma)$ is 
\begin{equation}\label{equation_quasibracket}
\brac{\Phi_\alpha}{\Psi_\beta}_{M_G(\Sigma)}
=\sum_{I,J=\wedge, \vee}\varepsilon(\alpha^I,\beta^J)\bi{\Phi^I_\alpha}{\Psi^J_\beta}
+\sum_{q\in\alpha\#\beta}\varepsilon_q(\alpha,\beta)B^q_{\Phi,\alpha,\Psi,\beta},
\end{equation}
where $B^q_{\Phi,\alpha,\Psi,\beta}\in\Func{M_G(\Sigma)}$ is defined by
$$B^q_{\Phi,\alpha,\Psi,\beta}=\bi{\so{\Phi}_\alpha}{\Ad_{\Hol_{\alpha*_q\beta}}\ta{\Psi}_\beta}.$$
\end{theorem}

%

A proof will be given in the next subsection. We shall first give some remarks and easy consequences of the theorem. 
\begin{remark}\label{remark_quasibracket}
\begin{enumerate}[(i)]
\item\label{item_remarkquasibracket1}
Let us show that the right-hand side of the above formula is anti-symmetric when $(\Phi, \alpha)$ and $(\Psi, \beta)$ are exchanged.

The first sum is anti-symmetric because $\varepsilon(\alpha^I, \beta^J)=-\varepsilon(\beta^J, \alpha^I)$. 
For the second sum, since $\varepsilon_q(\alpha, \beta)=-\varepsilon_q(\beta, \alpha)$,  it is sufficient to show $$B^q_{\Phi,\alpha,\Psi,\beta}=B^q_{\Psi,\beta,\Phi,\alpha}.$$ This is proved by the following computation, using the $\Ad$-invariance of $\bi{\cdot}{\cdot}$, Lemma \ref{lemma_taso} and the observation that the path $\beta(\alpha*_q\beta)^{-1}\alpha$ is homotopic to $\beta*_q\alpha$:
\begin{align*}
B^q_{\Phi,\alpha,\Psi,\beta}&=\bi{\so{\Phi}_\alpha}{\Ad_{\Hol_{\alpha*_q\beta}}\ta{\Psi}_\beta}=\bi{\Ad_{\Hol_\alpha}\ta{\Phi}_\alpha}{\Ad_{\Hol_{\alpha*_q\beta}}\Ad_{\Hol_\beta}^{-1}\so{\Psi}_\beta}\\
&=\bi{\Ad_{\Hol_{\beta*_q\alpha}}\ta{\Phi}_\alpha}{\so{\Psi}_\beta}=B^q_{\Psi,\beta,\Phi,\alpha}.
\end{align*}

\item\label{item_remarkquasibracket2}
By similar computations as above, one can obtain the following more general expressions of $B^q_{\Phi,\alpha,\Psi,\beta}$, which will be used in \S \ref{subsection_proofquasibracket} when we prove the theorem. For any fixed $I$ and $J$, we define a path $$\gamma=(\alpha^{\varepsilon_{I\wedge}})*_q(\beta^{\varepsilon_{J\vee}}),$$
where $\varepsilon_{I\wedge}, \varepsilon_{J\vee}=\pm 1$ are defined in Lemma \ref{lemma_formulation}. Then we have
$$
B^q_{\Phi,\alpha,\Psi,\beta}=\bi{\Phi^I_\alpha}{\Ad_{\Hol_\gamma}\Psi^J_\beta}.
$$

\item\label{item_remarkquasibracket3}
Given $\Phi,\Psi\in\Func{G}$, $\Phi_\alpha$ and $\Psi_\beta$ only depend on the homotopy classes of $\alpha$ and $\beta$, so the right-hand side of (\ref{equation_quasibracket}) should be invariant under homotopy. This fact will be established in the next subsection as a step in the proof of the theorem.
\end{enumerate}
\end{remark}

A simple consequence of Theorem \ref{theorem_quasibracket} is the following
\begin{corollary}\label{corollary_alphaalpha}
If $\alpha$ is a simple path (i.e. has no self-intersection), then
$$\brac{\Phi_\alpha}{\Psi_\alpha}=0\quad\forall\Phi, \Psi\in\Func{G}.$$
\end{corollary}

\begin{proof}
Let $\beta$ be a path homotopic to $\alpha$ such that $\alpha\#\beta=\emptyset$ and that the relative configuration of $\alpha$ and $\beta$ at endpoints, in the case where $\alpha$ is closed or non-closed, are as shown in the following local pictures respectively.

\begin{figure}[h] 
\centering \includegraphics[width=3in]{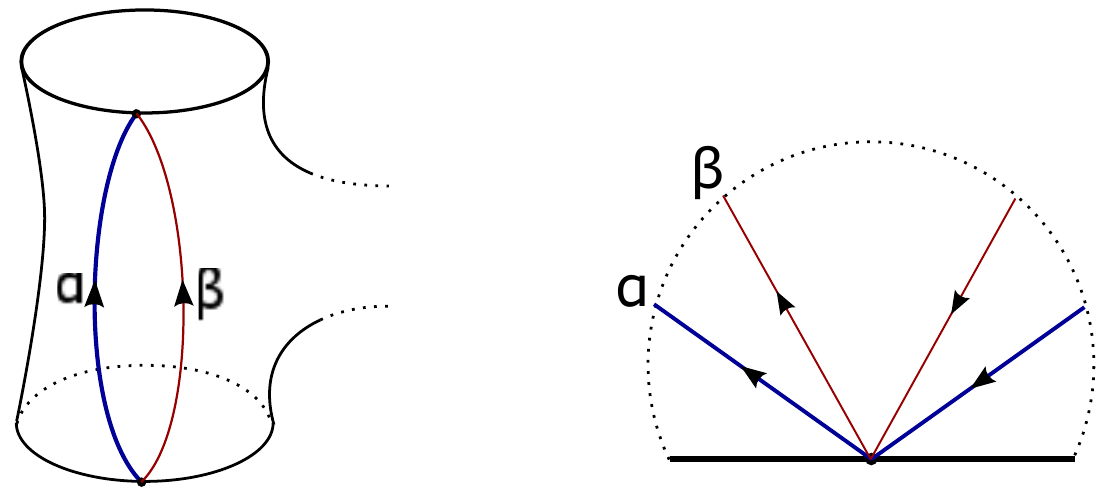} 
\end{figure}
 
Put $a=\Hol_\alpha=\Hol_\beta$. In the non-closed case we have $\varepsilon(\ta{\alpha},\ta{\beta})=\frac{1}{2}$, $\varepsilon(\so{\alpha},\so{\beta})=-\frac{1}{2}$ and $\varepsilon(\so{\alpha},\ta{\beta})=\varepsilon(\ta{\alpha},\so{\beta})=0$. So Theorem \ref{theorem_quasibracket} gives
$$
\brac{\Phi_\alpha}{\Psi_\alpha}=\brac{\Phi_\alpha}{\Psi_\beta}=\frac{1}{2}\bi{\ta{\Phi}(a)}{\ta{\Psi}(a)}-\frac{1}{2}\bi{\so{\Phi}(a)}{\so{\Psi}(a)},
$$
which vanishes because $\so{\Phi}(a)=\Ad_a\ta{\Phi}(a)$, $\so{\Psi}(a)=\Ad_a\ta{\Psi}(a)$ and the scalar product $\bi{\cdot}{\cdot}$ is $\Ad$-invariant. In the closed case the argument is similar.
\end{proof}

We also get the following particular case of Theorem \ref{theorem_quasibracket} by straightforward computations using Example \ref{example_caseGL}. Here $i(\alpha,\beta)$ is the algebraic intersection number defined in the previous subsection.
\begin{corollary}\label{corollary_quasibracketGL}
Assume $G=\GLR$ and the invariant scalar product is $\bi{x}{y}=\Tr(xy)$. For a path $\alpha$ on $\Sigma$, we set $\alpha_{ij}:=\Phi_{ij}\circ\Hol_\alpha$, where $\Phi_{ij}(g)=g_{ij}$ is the $(i,j)$-matrix entry function. Let $\alpha$ and $\beta$ be paths in general position. Then  
\begin{align}
&\brac{\alpha_{ij}}{\beta_{kl}}_{M_G(\Sigma)}=\sum\limits_{q\in\alpha\#\beta}\varepsilon_q(\alpha,\beta)(\alpha*_q\beta)_{il}(\beta*_q\alpha)_{kj}\label{equation_quasibracketGL}\\
+\varepsilon(&\alpha^\wedge,\beta^\wedge)\alpha_{kj}\beta_{il}+\varepsilon(\alpha^\vee,\beta^\vee)\alpha_{il}\beta_{kj}+\delta_{il}\varepsilon(\alpha^\wedge,\beta^\vee)(\beta\alpha)_{kj}+\delta_{jk}\varepsilon(\alpha^\vee,\beta^\wedge)(\alpha\beta)_{il}.\nonumber
\end{align}
Here $\delta_{ij}=0,1$ is the Kronecker delta.
\end{corollary}

\begin{remark}
Corollary \ref{corollary_quasibracketGL} implies that all the $\alpha_{ij}$'s generates a quasi-Poisson subalgebra $\mathcal{A}_n(\Sigma)\subset\Func{M_{\GLR}(\Sigma)}$.  $\mathcal{A}_n(\Sigma)$ can be described intrinsically as the commutative algebra over $\mathbb{R}$ generated by all symbols of the form $\alpha_{ij}$ (where $1\leq i, j\leq n$ and $\alpha$ is a non-trivial element in $\pi_1(\Sigma)$) with relations $\sum_{j=1}^n\alpha_{ij}\beta_{jk}=(\alpha\beta)_{ik}$. This algebra first appears in \cite{massuyeau-turaev} (where $b=1$).
\end{remark}

More consequences of the theorem can be found in the author's thesis \cite{nie_these}. For instance, if $\Phi, \Psi$ are invariant by conjugation and $\alpha, \beta$ are closed, then Theorem \ref{theorem_quasibracket} recovers Goldman's formula \cite{goldman_invariant}, hence gives a new proof of the latter. The flow generated by the vector field $P^\sharp(d\Phi_\alpha)$ (where $P$ is the quasi-Poisson tensor on $M_G(\Sigma)$) can also be explicitly determined when $\alpha$ is simple.

\subsection{Proof of Theorem \ref{theorem_quasibracket}}\label{subsection_proofquasibracket}
We shall use the following terminology in the course of proof. Let $\alpha$ and $\beta$ be paths in general position on $\Sigma$. We say that \emph{the theorem is true for the triple $(\Sigma,\alpha, \beta)$} if formula (\ref{equation_quasibracket}) holds for any choice of $\Phi, \Psi\in\Func{G}$. 

The strategy of proof is to successively reduce to simpler cases. Let us begin with some easy reductions.
\begin{lemma}\label{lemma_firstreduction}
If the theorem is true for a triple $(\Sigma, \alpha, \beta)$, then it is true for the following triples as well:
\begin{enumerate}[(i)]
\item\label{item_firstreductioni}
$(\Sigma, \beta, \alpha)$;
\item\label{item_firstreductionii}
$(\overline{\Sigma}, \alpha, \beta)$, where $\overline{\Sigma}$ denotes $\Sigma$ with the opposite orientation\footnote{$M_G(\Sigma)$ and $M_G(\overline\Sigma)$ are the same as $G^b$-manifolds, but with opposite quasi-Poisson tensors and moment maps.};
\item\label{item_firstreductioniii}
$(\Sigma, \alpha, \beta^{-1})$, $(\Sigma,\alpha^{-1}, \beta)$ and $(\Sigma, \alpha^{-1}, \beta^{-1})$.
\end{enumerate}
\end{lemma}
\begin{proof}
For (\ref{item_firstreductioni}) this is essentially Remark \ref{remark_quasibracket} (\ref{item_remarkquasibracket1}). 

As for (\ref{item_firstreductionii}), this is because on one hand we have $\brac{\Phi_\alpha}{\Psi_\beta}_{M_G(\overline\Sigma)}=-\brac{\Phi_\alpha}{\Psi_\beta}_{M_G(\Sigma)}$ because the quasi-Poisson tensors on $\Sigma$ and $\overline\Sigma$ are opposite; and on the other hand, under the reversed orientation, $\varepsilon(\alpha^I,\beta^J)$ and $\varepsilon_q(\alpha,\beta)$ are $-1$ times of the old ones. 

For the triples in (\ref{item_firstreductioniii}), in view of (\ref{item_firstreductioni}), it is sufficient to prove for $(\Sigma, \alpha, \beta^{-1})$. Put $\hat{\Psi}(g):=\Psi(g^{-1})$. By hypothesis we have
\begin{align*}
\brac{\Phi_\alpha}{\Psi_{\beta^{-1}}}_{M_G(\Sigma)}&=\brac{\Phi_\alpha}{\hat\Psi_\beta}_{M_G(\Sigma)}\\
&=\sum_{I,J}\varepsilon(\alpha^I,\beta^J)\bi{\Phi^I_\alpha}{\hat\Psi^J_\beta}+\sum_{q\in\alpha\#\beta}\varepsilon_q(\alpha,\beta)B^q_{\Phi,\alpha,\hat\Psi,\beta}.
\end{align*}

Using Lemma \ref{lemma_taso} (\ref{item_tasoii}), we get
\begin{align*}
\Ad_{\Hol_{\alpha*_q\beta}}\ta{\hat\Psi}_\beta=-\Ad_{\Hol_{\alpha*_q\beta}\Hol_{\beta^{-1}}}\ta{\Psi}(\Hol_{\beta^{-1}})=-\Ad_{\Hol_{\alpha*_q\beta^{-1}}}\ta{\Psi}_{\beta^{-1}},
\end{align*} 
which implies $B^q_{\Phi,\alpha,\hat\Psi,\beta}=-B^q_{\Phi,\alpha,\Psi,\beta^{-1}}$, hence
$$\varepsilon_q(\alpha,\beta)B^q_{\Phi,\alpha,\hat\Psi,\beta}=\varepsilon_q(\alpha,\beta^{-1})B^q_{\Phi,\alpha,\Psi,\beta^{-1}}.$$

Similarly, for any $I,J=\wedge, \vee$ we have
$$
\varepsilon(\alpha^I,\beta^J)\bi{\Phi^I_\alpha}{\hat\Psi^J_\beta}=\varepsilon(\alpha^I,(\beta^{-1})^{J'})\bi{\Phi^I_\alpha}{\Psi^{J'}_{\beta^{-1}}},
$$
where $J'$ denotes the symbol opposite to $J$. Therefore, we have the required equality
\begin{align*}
\brac{\Phi_\alpha}{\Psi_{\beta^{-1}}}_{M_G(\Sigma)}=\sum_{I,J}\varepsilon(\alpha^I,(\beta^{-1})^J)\bi{\Phi^I_\alpha}{\Psi^J_{\beta^{-1}}}
+\sum_{q\in\alpha\#\beta^{-1}}\varepsilon_q(\alpha,\beta^{-1})B^q_{\Phi,\alpha,\Psi,\beta^{-1}}.
\end{align*}

\end{proof}

The main ingredients of our proof of Theorem \ref{theorem_quasibracket} are the following reductions.
\begin{proposition}\label{proposition_reducing}
\begin{enumerate}[(i)]
\item\label{item_homotopy}
Let $\alpha'$ and $\beta'$ be two   paths in general position which are endpoints-fixing-homotopic to $\alpha$ and $\beta$ respectively. Then the right-hand side of (\ref{equation_quasibracket}) gives the same function when $\alpha$ and $\beta$ are replace by $\alpha'$ and $\beta'$, respectively. In particular, if the theorem is true for the triple $(\Sigma, \alpha, \beta)$, then it is also true for
$(\Sigma,\alpha', \beta')$.
\item\label{item_composition}
Let $\alpha_1,\cdots,\alpha_r$ and $\alpha$ be   paths such that $\alpha$ and each $\alpha_i$ are in general position with $\beta$, and $\alpha$ is endpoints-fixing-homotopic to the composition $\alpha_1\cdots\alpha_r$, where we assume that each $\alpha_i$ ends at the point where $\alpha_{i+1}$ starts from. If the theorem is true for each of the triples $(\Sigma, \alpha_1, \beta),\cdots, (\Sigma, \alpha_r, \beta)$, then it is also true for $(\Sigma, \alpha, \beta)$.
\item\label{item_reducingsplit} 
Let $\Delta$ be a splitting arc in $\Sigma$. Let $\tilde\alpha$, $\tilde\beta$ be   paths in general position on $\Sigma_\Delta$, and $\alpha$, $\beta$ be their images in $\Sigma$. If the theorem is true for the triple $(\Sigma_\Delta, \tilde\alpha, \tilde\beta)$, then it is also true for $(\Sigma, \alpha, \beta)$. 
\end{enumerate}
\end{proposition}

\begin{proof}[Proof of Proposition \ref{proposition_reducing} (\ref{item_homotopy})]
A straightforward generalization of results in \cite{goldman_invariant} \S 5 to surfaces with boundary shows that there exists a sequence of pairs of   paths in general position
$$(\alpha,\beta)=(\alpha_1,\beta_1),(\alpha_2,\beta_2),\cdots,(\alpha_r,\beta_r)=(\alpha',\beta')$$ 
such that $(\alpha_{i+1},\beta_{i+1})$ is obtained from $(\alpha_i,\beta_i)$ by applying one of the following moves ($\omega$1)-($\omega$4), as shown in Figure \ref{figure_monogon}-\ref{figure_bbigon}.

\begin{figure}[h]
\centering \includegraphics[width=1.8in]{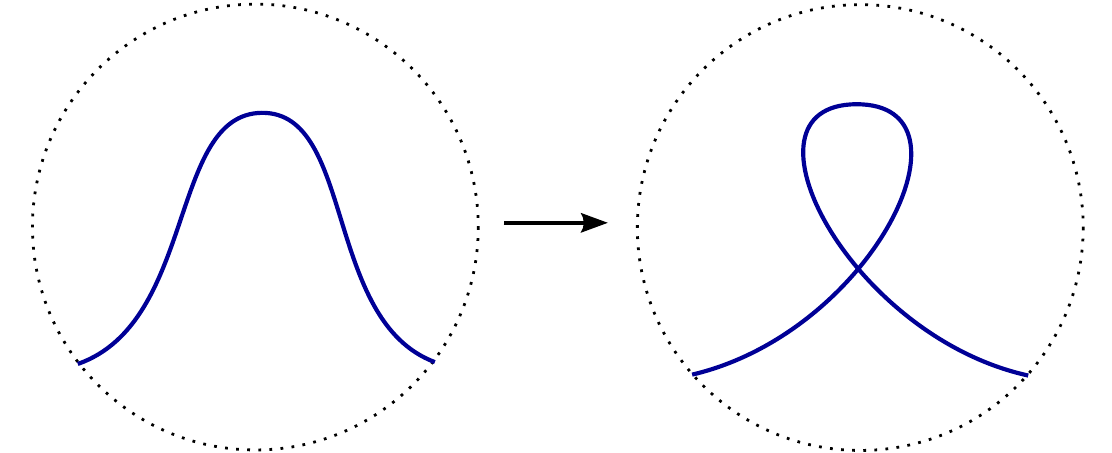}
\caption{($\omega$1): birth-death of monogons}
\label{figure_monogon}
\end{figure}

\begin{figure}[h] 
\centering \includegraphics[width=2.2in]{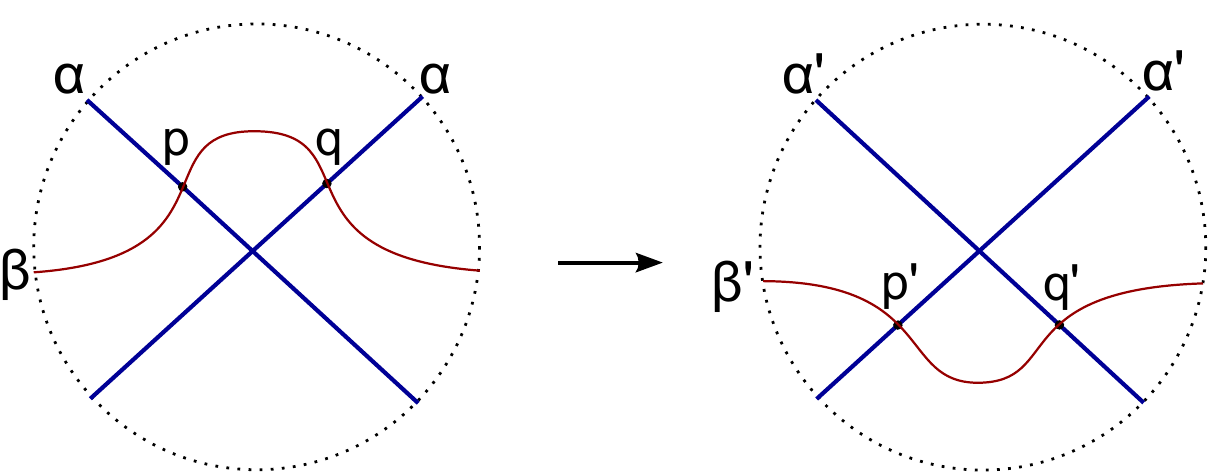} 
\hspace*{1cm}
\includegraphics[width=2.2in]{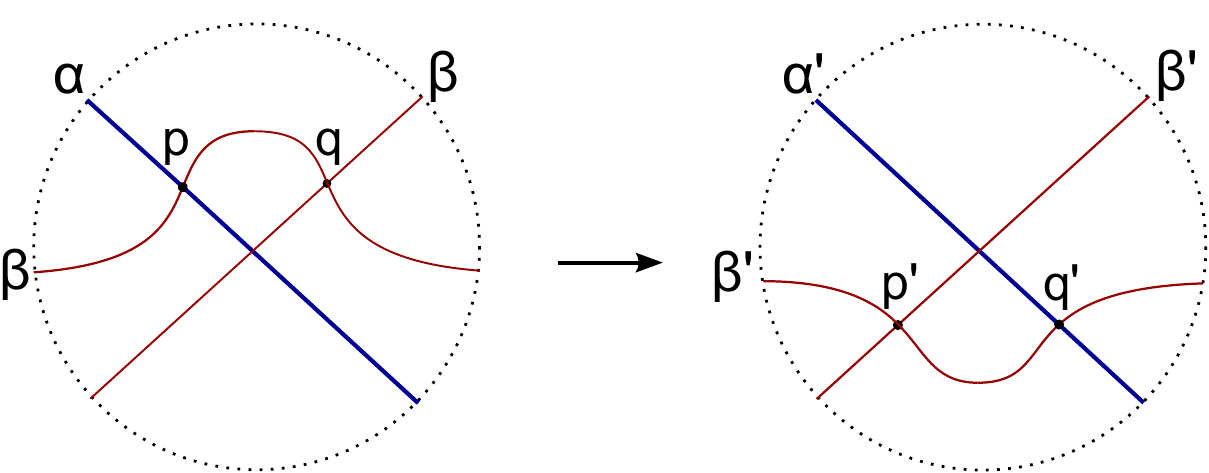} 
\caption{($\omega$2): jumping over a double point}
\label{figure_jump}
\end{figure}

\begin{figure}[h] 
\centering \includegraphics[width=2.2in]{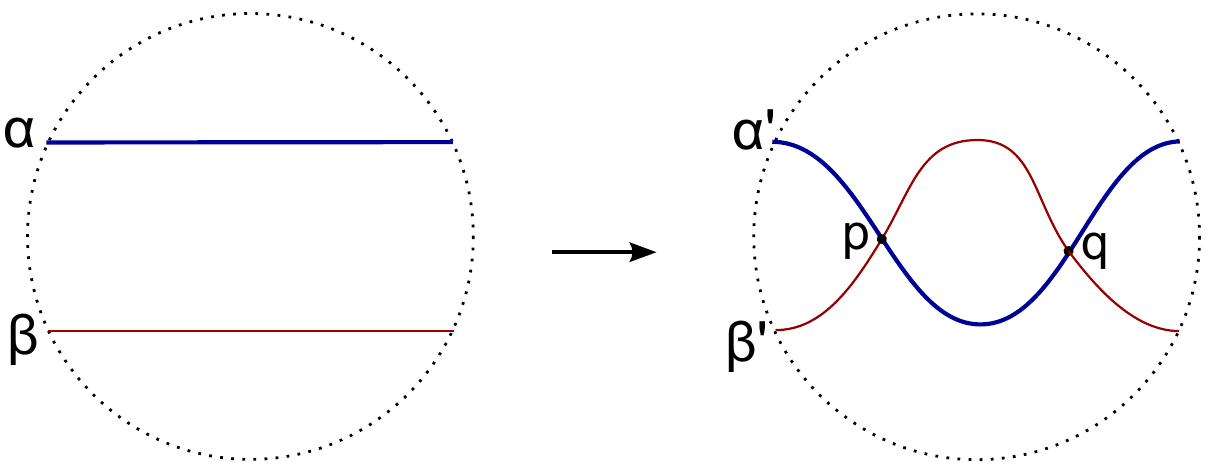} 
\caption{($\omega$3): birth-death of interior bigons}
\label{figure_ibigon}
\end{figure}

\begin{figure}[h] 
\centering \includegraphics[width=2.2in]{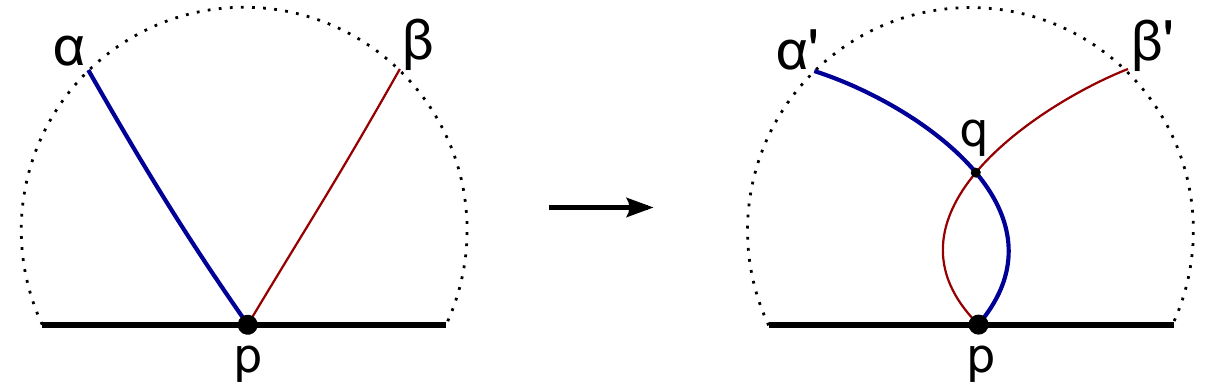} 
\caption{($\omega$4): 
birth-death of bigons with one vertex on the boundary}
\label{figure_bbigon}
\end{figure}

To apply a move ($\omega$j) to $(\alpha,\beta)$, first we find an open set $U$ in $\Sigma$ which is diffeomorphic to a disk (for $j=1,2,3$) or a half disk (for $j=4$), such that there are subintervals of $\alpha$ and $\beta$ in $U$ as shown in the pictures, then we replace these intervals by the new ones shown in the pictures.

We only need to prove that if $(\alpha',\beta')$ is obtained from $(\alpha,\beta)$ by one of the above moves, then replacing $\alpha$ and $\beta$ on the right-hand side of formula (\ref{equation_quasibracket}) by $\alpha'$ and $\beta'$ gives the same result.

Clearly, the move ($\omega$1) does not give rise to any change of the right-hand side of the formula. The next two moves ($\omega$2) and ($\omega$3) do not change the starting/ending directions $\alpha^I,\beta^J (I,J=\wedge, \vee)$, thus do not change the contribution from endpoints. Regarding the crossings, it can be shown (see p.293 of \cite{goldman_invariant} for details) that ($\omega$2) turns two points $p,q\in\alpha\#\beta$ into new ones $p',q'\in\alpha'\#\beta'$, such that 
$\alpha*_p\beta$ and $\alpha*_q\beta$ are homotopic to $\alpha'*_{p'}\beta'$ and $\alpha'*_{q'}\beta'$ respectively, and moreover $\varepsilon_p(\alpha,\beta)=\varepsilon_{p'}(\alpha',\beta')$, $\varepsilon_q(\alpha,\beta)=\varepsilon_{q'}(\alpha',\beta')$; whereas the move ($\omega$3) creates two new interior intersection points $p,q\in\alpha'\#\beta'$ such that $\alpha'*_p\beta'$ and $\alpha'*_q\beta'$ are homotopic to each other and $\varepsilon_p(\alpha',\beta')=-\varepsilon_q(\alpha',\beta')$. As a result, the formula remains unchanged under these two moves.

Finally, let us see what happens when applying the move ($\omega$4) to $(\alpha,\beta)$ near a marked point $p$. For two specific indices $I,J$ such that $\alpha^I, \beta^J\vdash p$, the move reverses the relative position of $\alpha^I$ and $\beta^J$, and creates a new interior intersection point $q\in\alpha'\#\beta'$. The only change of the right-hand side of formula (\ref{equation_quasibracket}) is that the term $\varepsilon(\alpha^I,\beta^J)\bi{\Phi^I_\alpha}{\Psi^J_\beta}$ becomes
$$
\varepsilon(\alpha'^I,\beta'^J)\bi{\Phi^I_\alpha}{\Psi^J_\beta}+\varepsilon_q(\alpha',\beta')B^q_{\Phi,\alpha',\Psi,\beta'}.
$$
We need to identify this with the original term. Observe that $$\varepsilon(\alpha^I,\beta^J)=-\varepsilon(\alpha'^I,\beta'^J)=\frac{1}{2}\varepsilon_q(\alpha',\beta'),$$ so it is sufficient to show
$$
B^q_{\Phi,\alpha',\Psi,\beta'}=\bi{\Phi^I_\alpha}{\Psi^J_\beta}.
$$
This follows from the expression of $B^q_{\Phi,\alpha',\Psi,\beta'}$ given in Remark \ref{remark_quasibracket} (\ref{item_remarkquasibracket2}): put $\gamma:=(\alpha'^{\varepsilon_{I\wedge}})*_q(\beta'^{\varepsilon_{J\vee}})$, then we have
$$
B^q_{\Phi,\alpha',\Psi,\beta'}=\bi{\Phi^I_\alpha}{\Ad_{\Hol_\gamma}\Psi^J_\beta}.
$$
But it is easy to see that $\gamma$ is the path which starts from $p$, runs along $\alpha$ until $q$, and then comes back to $p$ along $\beta$, hence is homotopically trivial. 
\end{proof}

\begin{proof}[Proof of Proposition \ref{proposition_reducing} (\ref{item_composition})]
It is sufficient to treat the $r=2$ case, from which the general case follows by recurrence. For brevity, we put 
$$u_1=\Hol_{\alpha_1},\quad u_2=\Hol_{\alpha_2},\quad u=\Hol_\alpha=u_1u_2,\quad v=\Hol_\beta.$$ 

We fix a point $m_0\in M_G(\Sigma)$, define  $u^{(1)},u^{(2)}: M_G(\Sigma)\rightarrow G$
by
$$
u^{(1)}(m)=u_1(m)u_2(m_0),\quad u^{(2)}(m)=u_1(m_0)u_2(m)\quad\forall m\in M_G(\Sigma),
$$
and define $\Phi_1, \Phi_2\in\Func{G}$ by
$$
\Phi_1(g)=\Phi(gu_2(m_0)),\quad \Phi_2(g)=\Phi(u_1(m_0)g)\quad\forall g\in G.
$$
Then the derivation of $\Phi_\alpha=\Phi(u)=\Phi(u_1u_2)\in\Func{M_G(\Sigma)}$ at the point $m_0$ is the sum of the derivations of $(\Phi_1)_{\alpha_1}=\Phi_1(u_1)$ and $(\Phi_2)_{\alpha_2}=\Phi_2(u_2)$. As a result, we get

\begin{align}
&\brac{\Phi_\alpha}{\Psi_\beta}(m_0)=\brac{(\Phi_1)_{\alpha_1}}{\Psi_\beta}(m_0)+\brac{(\Phi_2)_{\alpha_2}}{\Psi_\beta}(m_0)\label{equation_composition2}\\
=\sum_{I,J}&\varepsilon(\alpha_1^I,\beta^J)\bi{\Phi_1^I(u_1(m_0))}{\Psi^J(v(m_0))}+\sum_{q_1\in\alpha_1\#\beta}\varepsilon_{q_1}(\alpha_1,\beta)B^{q_1}_{\Phi_1,\alpha_1,\Psi,\beta}(m_0)
\nonumber\\
+\sum_{I,J}&\varepsilon(\alpha_2^I,\beta^J)\bi{\Phi_2^I(u_2(m_0))}{\Psi^J(v(m_0))}+\sum_{q_2\in\alpha_2\#\beta}\varepsilon_{q_2}(\alpha_2,\beta)B^{q_2}_{\Phi_2,\alpha_2,\Psi,\beta}(m_0).
\nonumber
\end{align}
The second equality is because of the hypothesis that the theorem is true for the triples $(\Sigma,\alpha_1, \beta)$ and  $(\Sigma,\alpha_2, \beta)$.

Our goal is to show that the right-hand side of (\ref{equation_composition2}) coincides with the right-hand side of formula (\ref{equation_quasibracket}) evaluated at $m_0$ for some $\alpha$ homotopic to $\alpha_1\alpha_2$.

Let us choose $\alpha$ in the following way. First we modify $\alpha_1$ and $\alpha_2$ by homotopy such that all starting/ending directions of the three paths $\alpha_1, \alpha_2, \beta$ are distinct. Let $p$ be the ending point of $\alpha_1$. Then we get $\alpha$ by smoothing the ``corner" of $\alpha_1\alpha_2$ within a small neighborhood $U$ of $p$, see Figure \ref{figure_smoothing}.

The smoothing does not change the crossings of $\alpha_1\alpha_2$ with $\beta$ outside $U$, whereas it creates up to two crossings $q_J$ ($J=\wedge, \vee$) in $U$, depending on the relative position of $\alpha_1^\vee$, $\alpha_2^\wedge$ and $\beta^J$. There are two cases:
\begin{enumerate}[(a)]
\item\label{item_compositiona} If $\beta^J$ does not lie on $p$, or if $\beta^J\vdash p$ and the directions $\ta{\alpha}_1$, $\so{\alpha}_2$ are on the same side of $\beta^J$ (i.e., both on the left or both on the right), then the smoothing creates no crossing near $\beta^J$.

\item\label{item_compositionb} If $\beta^J\vdash p$ and the directions $\ta{\alpha}_1$, $\so{\alpha}_2$ are on the two sides of $\beta^J$ respectively, then the smoothing creates a crossing $q_J\in U$. 
\end{enumerate}
\begin{figure}[h]
\centering \includegraphics[width=2.7in]{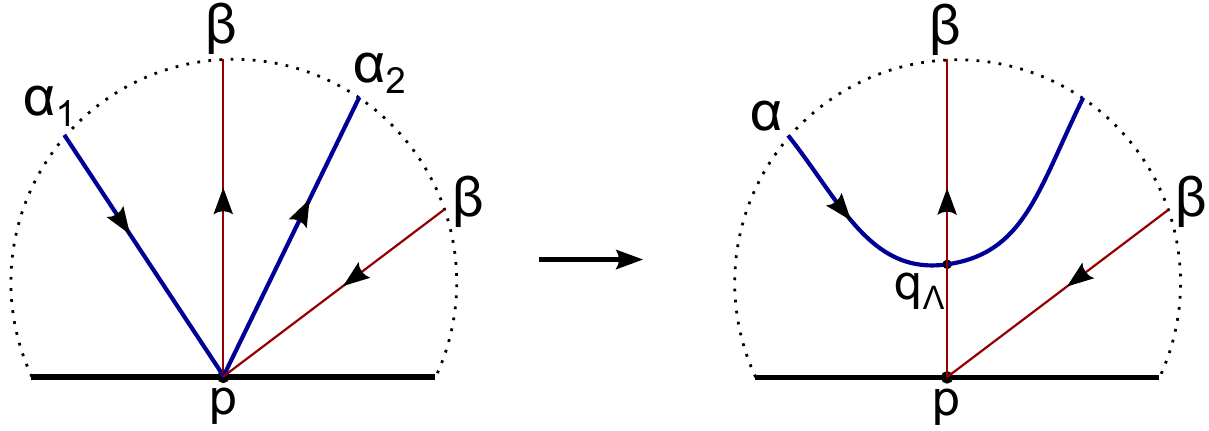} 
\caption{A typical local picture: here $\beta^\vee$ belongs to the case (\ref{item_compositiona}), and $\beta^\wedge$ to the case (\ref{item_compositionb})}\label{figure_smoothing}.
\end{figure}

Now we compute the right-hand side of (\ref{equation_composition2}). By Lemma \ref{lemma_taso} (\ref{item_tasoiii}), we have
the following identities of functions on $G$,
$$
\so{\Phi}_1(g)=\so{\Phi}(gu_2(m_0)),\quad \ta{\Phi}_1(g)=\Ad_{u_2(m_0)}\ta{\Phi}(gu_2(m_0)),
$$
and 
$$
\so{\Phi}_2(g)=\Ad_{u_1(m_0)}^{-1}\so{\Phi}(u_1(m_0)g),\quad \ta{\Phi}_2(g)=\ta{\Phi}(u_1(m_0)g).
$$
It follows that
\begin{equation}\label{equation_composition3}
\begin{array}{l}
\so{\Phi}_1(u_1(m_0))=\so{\Phi}(u(m_0)),\quad \ta{\Phi}_1(u_1(m_0))=\Ad_{u_2(m_0)}\ta{\Phi}(u(m_0)),\\
\so{\Phi}_2(u_2(m_0))=\Ad_{u_1(m_0)}^{-1}\so{\Phi}(u(m_0)),\quad \ta{\Phi}_2(u_2(m_0))=\ta{\Phi}(u(m_0)).
\end{array}
\end{equation}
Using (\ref{equation_composition3}), we find that certain terms in (\ref{equation_composition2}) coincide with those in  (\ref{equation_quasibracket}). Namely,
$$
\varepsilon(\alpha_1^\wedge,\beta^J)\bi{\Phi_1^\wedge(u_1(m_0))}{\Psi^J(v(m_0))}=\varepsilon(\alpha^\wedge,\beta^J)\bi{\Phi^\wedge(u(m_0))}{\Psi^J(v(m_0))},
$$
$$
\varepsilon(\alpha_2^\vee,\beta^J)\bi{\Phi_2^\vee(u_2(m_0))}{\Psi^J(v(m_0))}=\varepsilon(\alpha^\vee,\beta^J)\bi{\Phi^\vee(u(m_0))}{\Psi^J(v(m_0))},
$$
and
$$
\varepsilon_{q_i}(\alpha_i,\beta)B^{q_i}_{\Phi_i,\alpha_i,\Psi,\beta}(m_0)=\varepsilon_{q_i}(\alpha,\beta)B^{q_i}_{\Phi,\alpha,\Psi,\beta}(m_0),\quad\forall q_i\in\alpha_i\#\beta
$$
for $i=1,2$.

We see that each term in (\ref{equation_quasibracket}), except for those coming from the possible new crossings $q_\vee$ and $q_\wedge$, coincides with a term in (\ref{equation_composition2}). It remains to be shown that\begin{align*}
\varepsilon(\alpha_1^\vee,\beta^J)\bi{\Phi_1^\vee(u_1(m_0))}{\Psi^J(v(m_0))}+\varepsilon(\alpha_2^\wedge,\beta^J)\bi{\Phi_2^\wedge(u_2(m_0))}{\Psi^J(v(m_0))}\\
=\left\{ 
\begin{array}{ll}
0&\quad\mbox{in the case (\ref{item_compositiona})},\\
\varepsilon_{q_J}(\alpha,\beta)B^{q_J}_{\Phi,\alpha,\Psi,\beta}(m_0)&\quad\mbox{in the case (\ref{item_compositionb})}.
\end{array}
\right.
\end{align*}

In Case (\ref{item_compositiona}), $\varepsilon(\alpha_1^\vee,\beta^J)$ and  $\varepsilon(\alpha_1^\vee,\beta^J)$ are opposite, and it follows from (\ref{equation_composition3}) and Lemma \ref{lemma_taso} (\ref{item_tasoiii}) that
\begin{equation}\label{equation_composition4}
\ta{\Phi}_1(u_1(m_0))=\so{\Phi}_2(u_2(m_0)).
\end{equation}
Thus we get the required equality.

In Case (\ref{item_compositionb}), it is easy to see that
$$
\varepsilon(\ta{\alpha}_1,\beta^J)=\varepsilon(\so{\alpha_2},\beta^J)=\frac{1}{2}\varepsilon_{q_J}(\alpha,\beta).
$$
Using (\ref{equation_composition4}) and (\ref{equation_composition3}) again, we get
\begin{align*}
&\varepsilon(\alpha_1^\vee,\beta^J)\bi{\Phi_1^\vee(u_1(m_0))}{\Psi^J(v(m_0))}+\varepsilon(\alpha_2^\wedge,\beta^J)\bi{\Phi_2^\wedge(u_2(m_0))}{\Psi^J(v(m_0))}\\
&=\varepsilon_{q_J}(\alpha,\beta)\bi{\ta{\Phi}_1(u(m_0))}{\Ad_{u_2(m_0)}^{-1}\Psi^J(v(m_0))}.
\end{align*}
On the other hand, by Remark \ref{remark_quasibracket} (\ref{item_remarkquasibracket2}),
$$
B^{q_J}_{\Phi,\alpha,\Psi,\beta}=\bi{\ta{\Phi}_1(u)}{\Ad_{\Hol_\gamma}\Psi^J(v)},
$$  
where $\gamma:=(\alpha^{-1})*_{q_\vee}\beta$ if $J=\vee$, and $\gamma:=(\alpha^{-1})*_{q_\wedge}(\beta^{-1})$ if $J=\wedge$. Noting that $\gamma$ is homotopic to $\alpha_2$, we get $\Hol_\gamma=u_2$. This concludes the proof of the required equality.
\end{proof}

\begin{proof}[Proof of Proposition \ref{proposition_reducing} (\ref{item_reducingsplit})]

Suppose that the splitting arc $\Delta$ issues from a marked point $p\in\partial\Sigma$, and $p$ split into two marked point $p_1, p_2\in\partial\Sigma_\Delta$, where $p_1$ is on the left and $p_2$ on the right (see \S \ref{subsection_splitting}). For $i=1,2$, we let $\mu_i:M_G(\Sigma_\Delta)\rightarrow G$ denote the reversed boundary holonomy at $p_i$, and let $\chi_f^{(i)}$ denote the variation map (see (\ref{equation_variation}) for the definition) of $f\in\Func{M_G(\Sigma_\Delta)}$ with respect to the $G$-action on $M_G(\Sigma_\Delta)$ associated to $p_i$.

We consider $M:=M_G(\Sigma_\Delta)\cong M_G(\Sigma)$ as the same manifold via the splitting homoemorphism $R_\Delta$. If $\tilde\alpha$ is an   path on $\Sigma_\Delta$ and $\alpha$ is its image on $\Sigma$, then $\Hol_{\tilde\alpha}:M_G(\Sigma_\Delta)\rightarrow G$ and $\Hol_\alpha:M_G(\Sigma)\rightarrow G$ are same map, in particular $\Phi_{\tilde\alpha}=\Phi_\alpha\in\Func{M}$.

It follows from Theorem \ref{theorem_AMM} and the definition of fusion (\ref{equation_fusion}) that
$$ 
\brac{\Phi_\alpha}{\Psi_\beta}_{M_G(\Sigma)}=\brac{\Phi_{\tilde\alpha}}{\Psi_{\tilde\beta}}_{M_G(\Sigma_\Delta)}-\frac{1}{2}\bi{\chi^{(1)}_{\Phi_{\tilde\alpha}}}{\chi^{(2)}_{\Psi_{\tilde\beta}}}+\frac{1}{2}\bi{\chi^{(1)}_{\Psi_{\tilde\beta}}}{\chi^{(2)}_{\Phi_{\tilde\alpha}}}.
$$
The theorem is true for $(\Sigma_\Delta, \tilde\alpha,\tilde\beta)$ by hypothesis, hence
\begin{align*}
\brac{\Phi_{\tilde\alpha}}{\Psi_{\tilde\beta}}_{M_G(\Sigma_\Delta)}&=\sum_{I,J}\varepsilon(\tilde\alpha^I,\tilde\beta^J)\bi{\Phi^I_{\tilde\alpha}}{\Psi^J_{\tilde\beta}}+\sum_{q\in\tilde\alpha\#\tilde\beta}\varepsilon_q(\tilde\alpha,\tilde\beta)B^q_{\Phi,\tilde\alpha,\Psi,\tilde\beta}\\
&=\sum_{I,J}\varepsilon(\tilde\alpha^I,\tilde\beta^J)\bi{\Phi^I_\alpha}{\Psi^J_\beta}+\sum_{q\in\alpha\#\beta}\varepsilon_q(\alpha,\beta)B^q_{\Phi,\alpha,\Psi,\beta}.
\end{align*}
To prove that the theorem is true for $(\Sigma, \alpha, \beta)$, we need to show
\begin{align}\label{equation_reducingsplit}
&\sum_{I,J}\varepsilon(\tilde\alpha^I,\tilde\beta^J)\bi{\Phi^I_\alpha}{\Psi^J_\beta}-\frac{1}{2}\bi{\chi^{(1)}_{\Phi_{\tilde\alpha}}}{\chi^{(2)}_{\Psi_{\tilde\beta}}}+\frac{1}{2}\bi{\chi^{(1)}_{\Psi_{\tilde\beta}}}{\chi^{(2)}_{\Phi_{\tilde\alpha}}}\\
&=\sum_{I,J}\varepsilon(\alpha^I,\beta^J)\bi{\Phi^I_\alpha}{\Psi^J_\beta}.\nonumber
\end{align}
By Lemma \ref{lemma_formulation}, for $i=1,2$ we have
$$
\chi^{(i)}_{\Phi_{\tilde\alpha}}=\sum_{I: \tilde\alpha^I\vdash p_i}\varepsilon_{I\vee}\Phi^I_\alpha,
\quad
\chi^{(i)}_{\Psi_{\tilde\beta}}=\sum_{J: \tilde\beta^J\vdash p_i}\varepsilon_{J\vee}\Psi^J_\beta.
$$
Therefore, the left-hand side of (\ref{equation_reducingsplit}) equals
$$
\left(\sum_{I,J}\varepsilon(\tilde\alpha^I,\tilde\beta^J)-\sum_{\tilde\alpha^I\vdash p_1,\tilde\beta^J\vdash p_2}\frac{\varepsilon_{IJ}}{2}+\sum_{\tilde\alpha^I\vdash p_2,\tilde\beta^J\vdash p_1}\frac{\varepsilon_{IJ}}{2}\right)\bi{\Phi^I_\alpha}{\Psi^J_\beta}.
$$
Consider both sides of (\ref{equation_reducingsplit})  as a sum of four terms, corresponding to the four choices of $(I, J)$. Comparing the local pictures of $(\Sigma_\Delta, \tilde\alpha, \tilde\beta)$ and $(\Sigma, \alpha, \beta)$ near $p$, we see that for any $(I, J)$, 
\begin{itemize}
\item
If $\tilde\alpha^I\vdash p_1$ and $\tilde\beta^J\vdash p_2$, then $$\varepsilon(\tilde\alpha^I, \tilde\beta^J)=0,\quad
-\frac{\varepsilon_{IJ}}{2}=\varepsilon(\alpha^I, \beta^J);
$$
\item
If $\tilde\alpha^I\vdash p_2$ and $\tilde\beta^J\vdash p_1$,  then $$\varepsilon(\tilde\alpha^I, \tilde\beta^J)=0,\quad
\frac{\varepsilon_{IJ}}{2}=\varepsilon(\alpha^I, \beta^J);
$$
\item
If it is not the above two cases, then $\varepsilon(\tilde\alpha^I, \tilde\beta^J)=\varepsilon(\alpha^I, \beta^J)$.
\end{itemize}
This implies the required equality (\ref{equation_reducingsplit}).
\end{proof}

Using the above proposition, we will essentially reduce Theorem \ref{theorem_quasibracket} to the following three simplest situations, which can be verified directly.
\begin{proposition}\label{proposition_simplequasibracket}
\begin{enumerate}[(i)]
\item\label{item_disjointquasibracket}
If $\Sigma=\Sigma_1\sqcup\Sigma_2$ is a disjoint union, and $\alpha$, $\beta$ are contained in $\Sigma_1$, $\Sigma_2$, respectively, then the theorem is true for $(\Sigma,\alpha,\beta)$.

\item\label{item_holonomyquasibracket}
Let $\nu$ be an   loop which is homotopic to the reversed boundary loop of $\Sigma$ at a marked point $p$, and $\beta$ be any   path in general position with $\nu$. Then the theorem is true for $(\Sigma,\nu,\beta)$.

\item\label{item_uquasibracket}
Let $\alpha$ be an   paths on $\Sigma_{0,2}$ joining the the two marked points. Then $$\brac{\Phi_\alpha}{\Psi_{\alpha}}_{M_G(\Sigma_{0,2})}=0,\quad\forall \Phi,\Psi\in\Func{G}.$$ As a result, let $\alpha'$ be a path in general position with $\alpha$ and homotopic to $\alpha$, then the theorem is true for $(\Sigma_{0,2},\alpha,\alpha')$. 
\end{enumerate}
\end{proposition}
\begin{proof} 
(\ref{item_disjointquasibracket}) is clear. We proceed to show (\ref{item_holonomyquasibracket}).

Because of Proposition \ref{proposition_reducing} (\ref{item_homotopy}), we can modify $\nu$ by homotopy. Let us bring $\nu$ to a small neighborhood of the boundary, such that $\nu$ and $\beta$ have no crossings, and any $\beta^J\vdash p$ ($J=\wedge, \vee$) lies between $\nu^\vee$ and $\nu^\wedge$. Put $\delta_p(\beta^J)=1$ if $\beta^J\vdash p$, and otherwise $\delta_p(\beta^J)=0$. Then it is easy to see that the $\varepsilon(\nu^I,\beta^J)$'s are given by
\begin{equation}\label{equation_holonomyquasibracket1}
\varepsilon(\nu^I,\ta{\beta})=\delta_p(\ta{\beta}), \quad \varepsilon(\nu^I,\so{\beta})=-\delta_p(\so{\beta}), \mbox{ for } I=\wedge, \vee.
\end{equation}

Now we compute the quasi-Poisson bracket. Since $\mu:=\Hol_{\nu}$ is a component of the moment map $M_G(\Sigma)\rightarrow G^b$, (\ref{equation_diffofhol}) implies that\begin{align}\label{equation_holonomyquasibracket2}
&\brac{\Phi_\nu}{\Psi_\beta}_{M_G(\Sigma)}=-d\Phi_\nu(P^\sharp(d\Psi_\beta))=-\bi{\ta{\Phi}(\mu)}{\nu^*\theta(P^\sharp(d\Psi_\beta))}\\
&=\frac{1}{2}\bi{\ta{\Phi}(\mu)}{(1+\Ad_\mu^{-1})\chi_{\Psi_\beta}^p}=\frac{1}{2}\bi{\ta{\Phi}(\mu)+\so{\Phi}(\mu)}{\chi^p_{\Psi_\beta}},\nonumber
\end{align}
where $\chi^p_f$ denotes the variation function of $f\in\Func{M_G(\Sigma)}$ with respect to the $G$-action associated to $p$. By Lemma \ref{lemma_formulation}, we have $$\chi^p_{\Psi_\beta}=-\delta_p(\so{\beta})\so{\Psi}_\beta+\delta_p(\ta{\beta})\ta{\Psi}_\beta.$$
Inserting this into (\ref{equation_holonomyquasibracket2}) and using (\ref{equation_holonomyquasibracket1}), we conclude that
$$
\brac{\Phi_\alpha}{\Psi_\beta}_{M_G(\Sigma)}=\sum_{I,J}\varepsilon(\nu^I,\beta^J)\bi{\Phi^I(\mu)}{\Psi^J_\beta},
$$
which agrees with formula (\ref{equation_quasibracket}) because $\nu$ and $\beta$ have no crossings.

To prove (\ref{item_uquasibracket}), let $p_1$ and $p_2$ be the starting and ending point of $\alpha$, respectively, and $\beta$ be the boundary loop at $p_2$. Set $a:=\Hol_\alpha$ and $b:=\Hol_{\beta\alpha^{-1}}$.
 $$(a,b):M_G(\Sigma_{0,2})\overset\sim\rightarrow G\times G=D(G)$$ identifies $M_G(\Sigma_{0,2})$ with the double $D(G)$, on which the canonical quasi-Poisson bivector field $P$ is given in Example \ref{example_double}.

To prove the first assertion, it is sufficient to show that $$a^*\theta(P^\sharp(d\Phi_\alpha))=0.$$
This is done by straightforward computation. In fact, using $d\Phi_\alpha=d(\Phi\circ a)=\bi{\ta{\Phi}(a)}{a^*\theta}$ and the expression of $P$, we get
$$
P^\sharp(d\Phi_\alpha)=\frac{1}{2}\sum_i\bi{\ta{\Phi}(a)}{a^*\theta(e_i^{1,L})}e_i^{2,R}+\frac{1}{2}\sum_i\bi{\ta{\Phi}(a)}{a^*\theta(e_i^{1,R})}e_i^{2,L},
$$
whence $P^\sharp(d\Phi_\alpha)$ is tangent to the second factor of $D(G)$, and the required property follows.

Finally, we have already seen in the proof of Corollary \ref{corollary_alphaalpha} that the the theorem gives zero when applied to $(\Sigma_{0,2},\alpha,\alpha')$. Thus the theorem is true for $(\Sigma_{0,2},\alpha,\alpha')$.
\end{proof}

We are now in position to prove the theorem.
\begin{proof}[Proof of Theorem \ref{theorem_quasibracket}]
Since any   path on $\Sigma$ is homotopic to the composition of a number of simple   paths (this follows, e.g., from the presentation of the fundamental groupoid $\pi_1(\Sigma)$ given in \S \ref{subsection_surface}), applying Proposition \ref{proposition_reducing} (\ref{item_homotopy}), it is sufficient to prove that the theorem is true for any  $(\Sigma,\alpha,\beta)$ where $\alpha$ and $\beta$ are simple   paths in general position. 

We shall now do a further reduction to the case where $\alpha$ and $\beta$ have no crossings. If $\alpha\#\beta$ is non-empty, suppose that it consists of points $q_1,\cdots, q_r$ ordered by the orientation of $\alpha$. We shall decompose $\alpha$ into   paths $\alpha_0,\cdots,\alpha_r$ in the following way, such that each $\alpha_i$ is simple and has no crossing with $\beta$. Then the required reduction follows by applying Proposition \ref{proposition_reducing} (\ref{item_homotopy}) again. First, $\alpha_0$ is obtained from the path $\alpha*_{q_1}\beta$ by smoothing its corner at $q_1$ and homotoping it away from $\beta$. Then we obtain $\alpha_i$ ($1\leq i\leq r-1$) from $\beta^{-1}*_{p_i}\alpha*_{p_{i+1}}\beta$, and $\alpha_r$ from $(\beta^{-1})*_{p_r}\alpha$ in a similar way. See Figure \ref{figure_decomposition}.

\begin{figure}[h] 
\centering \includegraphics[width=2.8in]{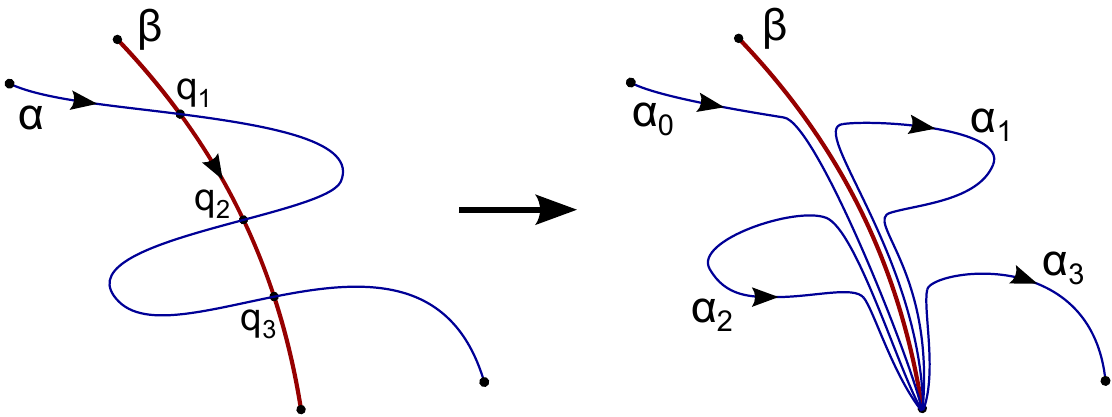} 
\caption{Decomposing $\alpha$ into paths which do not cross $\beta$.}\label{figure_decomposition} 
\end{figure}

Therefore, we only need to prove the theorem for any triple $(\Sigma,\alpha,\beta)$ where $\alpha$ and $\beta$ are simple   paths on $\Sigma$ in general position and without crossings. By Lemma \ref{lemma_firstreduction}, we could exchange the roles of $\alpha$ and $\beta$, reverse the orientation of $\Sigma$, or replace $\alpha$ and/or $\beta$ by their inverse. Performing these modifications if necessary, we can always bring the relative configuration of $\alpha$ and $\beta$ into the following three situations.

\begin{enumerate}[(a)]
\item\label{item_simple1}  The path $\alpha$ is a simple loop issuing from a marked point $p$. Furthermore, we assume that $\so{\alpha}$ is on the left of $\ta{\alpha}$, and also on the left of any $\beta^J\vdash p$ (where $J=\wedge, \vee$). 

\item\label{item_simple2} Both $\alpha$ and $\beta$ are embedded segments, and they share at most one endpoint. Furthermore, we assume that their common endpoint $p$, if exists, is the starting point of both $\alpha$ and $\beta$, and $\so{\alpha}$ is on the left of $\so{\beta}$.

\item\label{item_simple3} Both $\alpha$ and $\beta$ are embedded segments, and they both go from a marked point $p_1$ to another $p_2$.
\end{enumerate}

In case (\ref{item_simple1}), we can find a splitting arc $\Delta$ issuing from $p$ and homotopic to $\alpha$, such that $\Delta$ is disjoint from $\alpha$ and $\beta$ except at $p$ (see Figure \ref{figure_case12}). Thus $\alpha$ and $\beta$ lifts to   paths $\tilde\alpha$ and $\tilde\beta$ on the split surface $\Sigma_\Delta$. By Proposition \ref{proposition_reducing} (\ref{item_reducingsplit}), it is sufficient to prove the theorem for the triple $(\Sigma_\Delta,\tilde\alpha,\tilde\beta)$. But $\tilde\alpha$ is homotopic to the reversed boundary loop of $\Sigma_\Delta$ at the left split marked point, so we conclude by applying Proposition \ref{proposition_simplequasibracket} (\ref{item_holonomyquasibracket}).
\begin{figure}[h] 
\centering \includegraphics[width=3.6in]{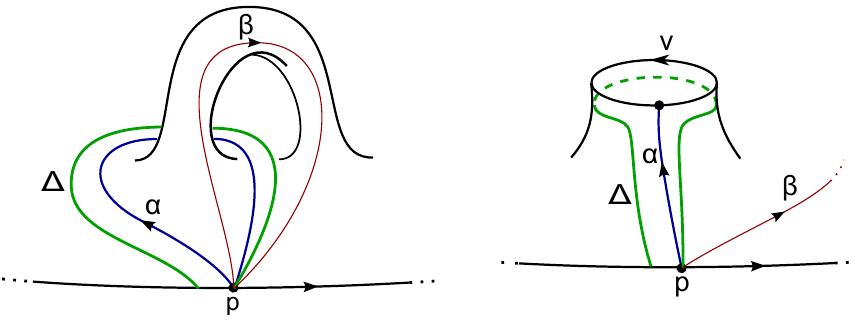} 
\caption{Typical situations in case (\ref{item_simple1}) (left) and case (\ref{item_simple2}) (right).}\label{figure_case12} 
\end{figure}

In case (\ref{item_simple2}), the argument is similar. Let $p$ be the starting point of $\alpha$, and $\nu$ the reversed boundary loop at the ending point of $\alpha$. We can find a splitting arc $\Delta$ issuing from $p$ and homotopic to $\alpha\nu\alpha^{-1}$, such that $\Delta$ is disjoint from $\alpha$ and $\beta$ except at $p$ (see Figure \ref{figure_case12}). The split surface $\Sigma_\Delta$ is the disjoint union of an annulus $\Sigma_{0,2}$ and a surface $\Sigma'$, whose number of boundary components is one less than $\Sigma$. The lifts $\tilde\alpha$ and $\tilde\beta$ of $\alpha$ and $\beta$ are contained in $\Sigma_{0,2}$ and $\Sigma'$, respectively. This time we conclude by applying Proposition \ref{proposition_simplequasibracket} (\ref{item_disjointquasibracket}).

Finally, in case (\ref{item_simple3}), we decompose $\beta$ into $(\beta\alpha^{-1})\alpha$. We homotope the loop  $\beta\alpha^{-1}$ into an   loop $\gamma$ which has no crossings with $\alpha$, and also homotope $\alpha$ into an   path $\alpha'$ in general position with $\alpha$. By Proposition \ref{proposition_reducing} (\ref{item_composition}) and the already proved case (\ref{item_simple2}), it is sufficient to prove the theorem for the triple $(\Sigma,\alpha,\alpha')$. But this follows from the same splitting as in case (\ref{item_simple2}) and Proposition \ref{proposition_simplequasibracket} (\ref{item_uquasibracket}). Thus the proof of Theorem \ref{theorem_quasibracket} is complete.
\end{proof}

\section{Quasi-Poisson brackets on cross-sections}\label{section_cross}
\subsection{Quasi-Poisson cross-sections}\label{subsection_cross}
In this subsection we recall the quasi-Poisson cross-section theorem from \cite{AKM}.

Assume that $G$ is compact and $\bi{\cdot}{\cdot}$ is a positive-definite invariant scalar product on $\g$. Given $g\in G$, let $H$ be the stabilizer of $g$ with respect to the conjugation action of $G$ on itself, and $\mathfrak{h}$ be its Lie algebra. Then $g$ has a neighborhood $U$ in $H$ such that $U$ is a cross-section for the conjugation action, in the sense that the map $G\times _H U\rightarrow G.U, (g, h)\mapsto ghg^{-1}$ is a diffeomorphism onto its image. 

Let $T\subset G$ a maximal torus, $\mathfrak{t}$ be its Lie algebra, and $\mathfrak{A}\subset\mathfrak{t}$ be a (closed) Weyl alcove. Without loss of generality we can assume $g\in \exp(\mathfrak{A})$. A standard choice of $U$ in this case is as follows. Let $\sigma$ be the open face of $\mathfrak{A}$ such that $g\in \exp(\mathfrak{A})$. The stabilizer of any element in $\sigma$ is the same subgroup $H=G_\sigma\subset G$.  Let $V_\sigma$ be the union of all open faces $\tau$ of $\mathfrak{A}$ such that $\overline\tau\supset\sigma$. Then we can take $U=G_\sigma.\exp(V_\sigma)$. In particular, we can take $U=\exp(\mathfrak{A}^\circ)$ if $g\in\exp(\mathfrak{A}^\circ)$.

Notice that with the above choice of $U$, the stabilizer of any $h\in U$ is contained in $G_\sigma=H$. It follows that $(\Ad_h-1)|_{\mathfrak{h}^\perp}$ is invertible.

Let $(M, P, \mu)$ be a Hamiltonian quasi-Poisson $G$-manifold. By equivariance of $\mu$, $L=\mu^{-1}(U)$ is a $H$-invariant smooth submanifold of $M$ and is a cross-section of the $G$-action in the sense that
$G\times _HL\rightarrow\mu^{-1}(G.U),\  (g, m)\mapsto g.m$ is a diffeomorphism.
It follows that there is a splitting 
$$
\T M|_L=TL\oplus (L\times \mathfrak{h}^\perp).
$$
Here we identify $(m, x)\in L\times \mathfrak{h}^\perp$  with $\rho_x(m)\in\T M|_L$.

\begin{theorem}[Cross-Section Theorem]\label{theorem_cross}
\begin{enumerate}[(i)]
\item\label{item_cross1}
There is a decomposition 
$$
P|_L=P_L+P_L^\perp,
$$ 
where $P_L\in\Gamma(\bigwedge^2\T L)$ and $P^\perp_L: L\rightarrow \bigwedge^2\mathfrak{h}^\perp$. Furthermore, $P_L^\perp(m)\in\bigwedge^2\mathfrak{h}^\perp\cong\bigwedge^2(\mathfrak{h}^\perp)^*$ (here we identify $\mathfrak{h}$ with $\mathfrak{h}^*$ via $\bi{\cdot}{\cdot}_{\mathfrak{h}^\perp}$) is the following skew-symmetric bilinear form on $\mathfrak{h}^\perp$:
$$
(x,y)\mapsto -\frac{1}{2}\bi{\left(\frac{\Ad_{\mu(m)}+1}{\Ad_{\mu(m)}-1}\right)x}{y}.
$$
\item
$(L, P_L, \mu|_L)$ is a Hamiltonian quasi-Poisson $H$-manifold.
\end{enumerate}
\end{theorem}

We will consider cross-sections of the quasi-Poisson $G^b$-manifold $M_G(\Sigma)$. From now on we put
$$
L=\bigcap_{i=1}^{b}\mu_i^{-1}(U)\subset M_G(\Sigma).
$$
Then $L$ is a smooth $H^b$-invariant submanifold of $M_G(\Sigma)$. The above theorem implies that there is a bivector field $P_L$ on $L$ so that $(L, P_L, (\mu_1|_L,\cdots, \mu_b|_L))$ is a Hamiltonian quasi-Poisson $H^b$-manifold. 

\subsection{The quasi-Poisson bracket formula for cross-sections}\label{subsection_crossbracket}
We shall deduce from Theorem \ref{theorem_quasibracket} and the Cross-Section Theorem a formula for quasi-Poisson brackets of functions of the form $\Phi_\alpha|_L$ on $L$.
This will involve an invertible linear map $\Op_h: \g\rightarrow\g$ depending on a parameter $h\in U$ defined by
$$
\Op_h=\pr_\mathfrak{h}+\frac{2}{1-\Ad_{h}}\pr_{\mathfrak{h}^\perp},
$$ 
where $\pr_\mathfrak{h}$ and $\pr_{\mathfrak{h}^\perp}$ are projections of $\g$ onto $\mathfrak{h}$ and $\mathfrak{h}^\perp$.

It is easy to see that the transpose of $\Op_h$ is 
$$
\Op_h^\top=\pr_\mathfrak{h}+\frac{2}{1-\Ad_{h}^{-1}}\pr_{\mathfrak{h}^\perp},
$$
and we have the identity
$$
\Op_h+\Op_h^\top=2.
$$


\begin{theorem}\label{theorem_crossbracket}
Let $\{\cdot,\cdot\}_L$ be the quasi-Poisson bracket defined by the quasi-Poisson structure of $L$, then
\begin{equation}\label{equation_bracket}
\brac{\Phi_\alpha}{\Psi_\beta}_{L}=\sum_{I,J=\wedge, \vee}\varepsilon(\alpha^I,\beta^J)A^{IJ}_{\Phi,\alpha,\Psi,\beta}+\sum_{q\in\alpha\#\beta}\varepsilon_q(\alpha,\beta)B^q_{\Phi,\alpha,\Psi,\beta},
\end{equation}
where $\Phi_\alpha$, $\Psi_\beta$ and $B^q_{\Phi,\alpha,\Psi,\beta}$ are the same as in Theorem \ref{theorem_quasibracket}, but restricted to $L$ here. $A^{IJ}_{\Phi,\alpha,\Psi,\beta}\in\Func{L}$ is defined as follows. If $\alpha^I$ and $\beta^J$ are at the same marked point $p_i$ then we define
$$
A^{IJ}_{\Phi,\alpha,\Psi,\beta}=\left\{
\begin{array}{ll}
\bi{\Op_{\mu_i}\Phi^I_\alpha}{\Psi^J_\beta}&\mbox{ if $\alpha^I$ is on the left of $\beta^J$ at $p_i$,}\\
 & \\
\bi{\Phi^I_\alpha}{\Op_{\mu_i}\Psi^J_\beta}&\mbox{ if $\alpha^I$ is on the right of $\beta^J$ at $p_i$;}
\end{array}
\right.
$$
otherwise we set $A^{IJ}_{\Phi,\alpha,\Psi,\beta}=0$.
\end{theorem}
By ``$\alpha^I$ is on the left of $\beta^J$ at $p_i$" is meant the following local picture, where we identify a neighborhood of $p_i$ with the upper-half plan in an orientation-preserving manner.
\begin{figure}[h] 
\centering \includegraphics[width=1.3in]{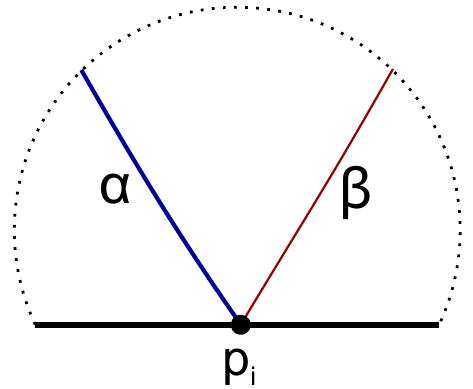} 
\end{figure}

\begin{proof}By definition of the quasi-Poisson tensor on $L$, 
\begin{align}\label{equation_bracket1}
\brac{\Phi_\alpha}{\Psi_\beta}_L&=P_L(d\Phi_\alpha, d\Psi_\beta)=P|_L(d\Phi_\alpha|_\mathfrak{h}, d\Psi_\beta|_\mathfrak{h})-P_L^\perp(d{\Phi_\alpha}|_{\mathfrak{h}^\perp}, d{\Psi_\beta}|_{\mathfrak{h}^\perp})\\
&=\{\Phi_\alpha,\Psi_\beta\}_{M_G(\Sigma)}|_L-P_L^\perp(d{\Phi_\alpha}|_{\mathfrak{h}^\perp}, d{\Psi_\beta}|_{\mathfrak{h}^\perp})\nonumber
\end{align}
 Let $\chi^{(i)}_{\Phi_\alpha}, \chi^{(i)}_{\Psi_\beta}: M_G(\Sigma)\rightarrow\g$ denote the variation maps (see (\ref{equation_variation}) for the definition) of $\Phi_\alpha$ and $\Psi_\beta$ with respect to the action of the $i$-th component of $G^b$ on $M_G(\Sigma)$. Using the expression of $P_L^\perp$ given in Theorem \ref{theorem_cross} and Lemma \ref{lemma_formulation}, we get
\begin{align*}
-P_L^\perp(d{\Phi_\alpha}|_{\mathfrak{h}^\perp}, d{\Psi_\beta}|_{\mathfrak{h}^\perp})&=\frac{1}{2}\sum_{i=1}^b\bi{\left(\frac{\Ad_{\mu_i}+1}{\Ad_{\mu_i}-1}\right)\pr_{\mathfrak{h}^\perp}(\chi_{\Phi_\alpha}^{(i)})}{\pr_{{\mathfrak{h}^\perp}}(\chi_{\Psi_\beta}^{(i)})}\\
&=\frac{1}{2}\sum_{i=1}^b\sum_{I,J}\varepsilon_{IJ}\bi{\left(\frac{\Ad_{\mu_i}+1}{\Ad_{\mu_i}-1}\right)\pr_{\mathfrak{h}^\perp}(\Phi^I_\alpha)}{\Psi^J_\beta},
\end{align*}
where for any fixed $i$ the summation ``$\sum_{I,J}$" runs over symbols $I, J=\wedge, \vee$ such that both $\alpha^I$ and $\beta^J$ lie on $p_i$

Inserting the the above equality and Theorem \ref{theorem_quasibracket} into (\ref{equation_bracket1}), we get
\begin{align}\label{equation_bracket2}
&\brac{\Phi_\alpha}{\Psi_\beta}_L=\sum_{q\in\alpha\#\beta}\varepsilon_q(\alpha,\beta)B^q_{\Phi,\alpha,\Psi,\beta}\\
+\sum_{i=1}^b&\sum_{I,J}\bi{\left[ \varepsilon_i(\alpha^I,\beta^J)+\frac{1}{2}\varepsilon_{IJ}\left( \frac{\Ad_{\mu_i}+1}{\Ad_{\mu_i}-1}\right)\pr_{{\mathfrak{h}^\perp}}\right] \Phi^I_\alpha}{\Psi^J_\beta},\nonumber
\end{align}
where $\varepsilon_i(\alpha^I,\beta^J)=0, \pm\frac{1}{2}$ is the oriented intersection number of $\alpha^I$ and $\beta^J$ at $p_i$. Namely, $\varepsilon_i(\alpha^I, \beta^J)=\varepsilon(\alpha^I,\beta^J)$ if both $\alpha^I$ and $\beta^J$ lie on $p_i$ and $\varepsilon_i(\alpha^I, \beta^J)=0$ otherwise.

It is easy to see that if $\alpha^I$ and $\beta^J$ both lie on $p_i$ and $\alpha^I$ is on the left of $\beta^J$, then $\varepsilon_i(\alpha^I,\beta^J)=-\frac{1}{2}\varepsilon_{IJ}$. This implies
\begin{align*}
\varepsilon_i(\alpha^I,\beta^J)+\frac{1}{2}\varepsilon_{IJ}\left( \frac{\Ad_{\mu_i}+1}{\Ad_{\mu_i}-1}\right)\pr_{\mathfrak{h}^\perp}&=\varepsilon_i(\alpha^I,\beta^J)\left(\pr_\mathfrak{h}+\pr_{\mathfrak{h}^\perp}-\frac{\Ad_{\mu_i}+1}{\Ad_{\mu_i}-1}\pr_{\mathfrak{h}^\perp}\right)\\
&=\varepsilon_i(\alpha^I,\beta^J)\Op_{\mu_i}.
\end{align*}
Similarly, if $\alpha^I$ is on the right of $\beta^J$,  then we have $\varepsilon_i(\alpha^I,\beta^J)=\frac{1}{2}\varepsilon_{IJ}$ and 
$$
\varepsilon_i(\alpha^I,\beta^J)+\frac{1}{2}\varepsilon_{IJ}\left( \frac{\Ad_{\mu_i}+1}{\Ad_{\mu_i}-1}\right)\pr_{\mathfrak{h}^\perp}=\varepsilon_i(\alpha^I,\beta^J)\Op_{\mu_i}^\top.
$$
Inserting these into (\ref{equation_bracket2}), we get the required equality (\ref{equation_bracket}).

\end{proof}

\bibliographystyle{plain} 
\bibliography{surface_groupoid}

\begin{thebibliography}{10}

\bibitem{AKM}
A.~Alekseev, Y.~Kosmann-Schwarzbach, and E.~Meinrenken.
\newblock Quasi-{P}oisson manifolds.
\newblock {\em Canad. J. Math.}, 54(1):3--29, 2002.

\bibitem{AMM}
Anton Alekseev, Anton Malkin, and Eckhard Meinrenken.
\newblock Lie group valued moment maps.
\newblock {\em J. Differential Geom.}, 48(3):445--495, 1998.

\bibitem{atiyah-bott}
M.~F. Atiyah and R.~Bott.
\newblock The {Y}ang-{M}ills equations over {R}iemann surfaces.
\newblock {\em Philos. Trans. Roy. Soc. London Ser. A}, 308(1505):523--615,
  1983.

\bibitem{fock-goncharov}
Vladimir Fock and Alexander Goncharov.
\newblock Moduli spaces of local systems and higher {T}eichm{\"u}ller theory.
\newblock {\em Publ. Math. Inst. Hautes {\'E}tudes Sci.}, (103):1--211, 2006.

\bibitem{goldman_nature}
William~M. Goldman.
\newblock The symplectic nature of fundamental groups of surfaces.
\newblock {\em Adv. in Math.}, 54(2):200--225, 1984.

\bibitem{goldman_invariant}
William~M. Goldman.
\newblock Invariant functions on {L}ie groups and {H}amiltonian flows of
  surface group representations.
\newblock {\em Invent. Math.}, 85(2):263--302, 1986.

\bibitem{labourie_swapping}
F.~{Labourie}.
\newblock {Goldman Algebra, Opers and the Swapping Algebra}.
\newblock {\em ArXiv e-prints}, December 2012.

\bibitem{lawton}
Sean Lawton.
\newblock Poisson geometry of {${\rm SL}(3,\Bbb C)$}-character varieties
  relative to a surface with boundary.
\newblock {\em Trans. Amer. Math. Soc.}, 361(5):2397--2429, 2009.

\bibitem{libland-severa}
D.~{Li-Bland} and P.~{{\v S}evera}.
\newblock {Moduli spaces for quilted surfaces and Poisson structures}.
\newblock {\em ArXiv e-prints}, December 2012.

\bibitem{massuyeau-turaev}
G.~{Massuyeau} and V.~{Turaev}.
\newblock {Quasi-Poisson structures on representation spaces of surfaces}.
\newblock {\em ArXiv e-prints}, May 2012.

\bibitem{nie_these}
Xin Nie.
\newblock Espace des modules quasi-hamiltonien des connexions plates sur une
  surface, 2013.

\bibitem{severa}
P.~{{\v S}evera}.
\newblock {Moduli spaces of flat connections and Morita equivalence of quantum
  tori}.
\newblock {\em ArXiv e-prints}, June 2011.

\end{thebibliography}
\end{document}